\documentclass{amsart}
\usepackage{
 amsmath, 
 amsxtra, 
 amsthm, 
 amssymb, 
 booktabs,
 comment,
 longtable,
 mathrsfs, 
 mathtools, 
 multirow,
 stmaryrd,
 tikz-cd, 
 bbm,
 xr,
 color,
 xcolor, url}
\usepackage{amscd}
 \usepackage[normalem]{ulem}

 \usepackage{colonequals} 
 \usepackage[bbgreekl]{mathbbol}
\usepackage[all]{xy}
\usepackage[nobiblatex]{xurl}
\usepackage[colorlinks=true, linkcolor=magenta, urlcolor=blue, citecolor=blue]{hyperref}
\usepackage{geometry}
\newtheorem{theorem}{Theorem}[section]
\newtheorem{lemma}[theorem]{Lemma}
\newtheorem{conjecture}[theorem]{Conjecture}
\newtheorem{proposition}[theorem]{Proposition}

\newtheorem{defn}[theorem]{Definition}
\newtheorem{example}[theorem]{Example}

\newcommand\robout{\bgroup\markoverwith {\textcolor{blue}{\rule[0.5ex]{2pt}{0.4pt}}}\ULon}

\newtheorem{lthm}{Theorem} 

\theoremstyle{remark}
\newtheorem{remark}[theorem]{Remark}

\makeatletter
\newcommand{\mylabel}[2]{#2\def\@currentlabel{#2}\label{#1}}
\makeatother

\newcommand{\Gal}{\mathrm{Gal}}

\newcommand{\Hom}{\mathrm{Hom}}
\newcommand{\Symb}{\mathrm{Symb}}
\newcommand{\cG}{\mathcal{G}}
\newcommand{\SL}{\mathrm{SL}}

\newcommand{\cor}{\mathrm{cor}}
\newcommand{\ord}{\mathrm{ord}}
\newcommand{\ZZ}{\mathbb{Z}}
\newcommand{\CC}{\mathbb{C}}

\newcommand{\QQ}{\mathbb{Q}}
\newcommand{\Qp}{\mathbb{Q}_p}
\newcommand{\Fp}{\mathbb{F}_p}
\newcommand{\Zp}{\ZZ_p}

\newcommand{\Sel}{\mathrm{Sel}}
\newcommand{\res}{\mathrm{res}}
\newcommand{\coker}{\mathrm{coker}}
\newcommand{\rank}{\mathrm{rank}}
\newcommand{\cX}{\mathcal{X}}

\usepackage[OT2,T1]{fontenc}
\DeclareSymbolFont{cyrletters}{OT2}{wncyr}{m}{n}
\DeclareMathSymbol{\Sha}{\mathalpha}{cyrletters}{"58}
\DeclareMathSymbol\dDelta  \mathord{bbold}{"01}
\definecolor{Green}{rgb}{0.0, 0.5, 0.0}

\usepackage[utf8]{inputenc}
\numberwithin{equation}{section}
\author{Antonio Lei}
\address{Antonio Lei\newline Department of Mathematics and Statistics\\University of Ottawa\\
150 Louis-Pasteur Pvt\\
Ottawa, ON\\
Canada K1N 6N5}
\email{antonio.lei@uottawa.ca}

\author{Robert Pollack}
\address{Robert Pollack\newline Department of Mathematics\\The University of Arizona\\617 N. Santa Rita Ave. \\
Tucson\\ AZ 85721-0089\\USA}
\email{rpollack@arizona.edu}

\author{Naman Pratap}
\address{Naman Pratap\newline Centre for Mathematical Sciences\\ University of Cambridge\\
Wilberforce Road,
Cambridge \\ CB3 0WA,
United Kingdom}
\email{np637@cam.ac.uk}

\subjclass[2020]{11R23}
\keywords{Iwasawa invariants, Mazur--Tate elements, elliptic curves, additive primes}

\begin{document}

\begin{abstract}
We investigate the $\lambda$-invariants of Mazur--Tate elements of elliptic curves defined over the field of rational numbers at primes of additive reduction. We explain their growth and how these invariants relate to other better understood invariants depending on the potential reduction type. We give examples and a conjecture for the additive potentially supersingular case, supported by computational data from Sage in this setting. Further, we extend our results to $\lambda$-invariants of Mazur--Tate elements of cuspidal Hecke eigenforms associated with potentially ordinary $p$-adic Galois representations. 
\end{abstract}
\title[Iwasawa Invariants of Mazur--Tate elements of elliptic curves]{On the Iwasawa Invariants of Mazur--Tate elements of elliptic curves at additive primes}

\maketitle

\section{Introduction}\label{sec:intro}
Let $p$ be an odd prime, and $E$ an elliptic curve defined over $\QQ$, with $f_E$ the weight two cusp form of level $N_E$ attached to $E$. Mazur and Swinnerton-Dyer \cite{MSD74} constructed a $p$-adic $L$-function attached to $E$ when it has good ordinary reduction at $p$. The construction of $p$-adic $L$-functions has been extended to multiplicative and good supersingular primes in \cite{AmiceVelu} and \cite{VISIK}.
In the case of good ordinary and multiplicative primes, the $p$-adic $L$-functions constructed in these works belong to $\Zp[[T]]\otimes \Qp$, and thus have finitely many zeros on the $p$-adic open unit disk. Their Iwasawa invariants (which measure the $p$-divisibility and the number of zeros in the open unit disk) can be defined via the $p$-adic Weierstrass preparation theorem. 
At supersingular primes, the construction in \cite{AmiceVelu,VISIK} yields a pair of $p$-adic $L$-functions which do not necessarily lie in an Iwasawa algebra. Nonetheless, the works \cite{pollack03} and \cite{sprung} show that they can be decomposed into $p$-adic $L$-functions that lie in $\Zp[[T]]\otimes\Qp$ via a logarithmic matrix. In particular, Iwasawa invariants are defined for each of these $p$-adic $L$-functions. 

The central objects of the present article are Mazur--Tate elements attached to elliptic curves, which are constructed using modular symbols and intimately related to the aforementioned $p$-adic $L$-functions. Originally called \emph{modular elements} in \cite{MT}, they can be realized as  $\Theta_M(E)\in\QQ[\Gal(\QQ(\zeta_{M})/\QQ)]$, where  $M\geq 1$ is an integer and $\zeta_M$ is a primitive $M$-th root of unity. The element $\Theta_M(E)$ interpolates the $L$-values of $E$ twisted by Dirichlet characters on $\Gal(\QQ(\zeta_M)/\QQ)$, normalized by appropriate periods (in the original article of Mazur and Tate, only even characters were considered and $\Theta_M$ were constructed as elements in $\QQ[(\ZZ/M\ZZ)^\times/\{\pm1\}]$). We shall concentrate on the Mazur--Tate elements $\vartheta_n(E)$ that belong to $\QQ[\Gal(\QQ(\zeta_{p^n})/\QQ)]$, where $p$ is our fixed prime number and $n\ge0$ is an integer. Furthermore, we may regard $\vartheta_n(E)$ as an element of $\Zp[\Gal(\QQ(\zeta_{p^n})/\QQ)]$ after an appropriate normalisation. These elements satisfy a norm relation as $n$ varies, which can be derived from the action of Hecke operators on modular symbols. One can define Iwasawa invariants of these Mazur--Tate elements, which are intimately linked to the $p$-adic valuations of the $L$-values of $E$ twisted by Dirichlet characters of $p$-power conductor as a consequence of the aforementioned interpolation property. In cases where the construction of a $p$-adic $L$-function is known (i.e., when $E$ has good ordinary, good supersingular, or multiplicative reduction at $p$), one can relate these invariants to those of the $p$-adic $L$-function, see \cite{PW} and \S\ref{sec:known} below for further details.


\subsection{Notation}
Let $\QQ_\infty/\QQ$ denote the cyclotomic $\Zp$-extension of $\QQ$ with $\Gamma \colonequals\Gal(\QQ_\infty/\QQ) \cong \Zp$. We fix a topological generator $\gamma$ of $\Gamma$. Let $\Gamma_n\colonequals\Gamma^{p^n}$ for an integer $n\ge0$.
We write $k_n\colonequals \QQ_\infty^{\Gamma_n}$, which is a cyclic sub-extension of $\QQ_\infty/\QQ$ of degree $p^n$. Let $\mathcal{G}_n \colonequals \Gal(\QQ(\mu_{p^n})/\QQ)$ and $G_n\colonequals \Gal(k_n/\QQ)$. We define the Iwasawa algebra $\Lambda$ as $\displaystyle\varprojlim_{n}\Zp[G_n]$. We fix an isomorphism $\Lambda \cong \Zp[[T]]$ that sends $\gamma$ to $1+T$. The Teichm\"uller character is denoted by $\omega: (\ZZ/p\ZZ)^\times \to \Zp^\times$. In what follows, we write $\theta_{n,i}(E)$ for the $\omega^i$-isotypic component of $\vartheta_{n+1}(E)$. When $i=0$, we simply write $\theta_n(E)$. We use $L_p(E, \omega^i, T)$ to denote the $\omega^i$-isotypic component of the $p$-adic $L$-function of $E$ whenever its construction is possible; for more details, see \S~\ref{ssec: MT and Lp}.

\subsection{Known results}\label{sec:known}
The connection of Iwasawa invariants of Mazur--Tate elements to Iwasawa invariants of $p$-adic $L$-functions is easiest to see in the case of an elliptic curve $E/\QQ$ and a prime $p$ of multiplicative reduction.  In this case, the $p$-adic $L$-function of $E$ is nothing other than the inverse limit of $\theta_n(E)/a_p(E)^{n+1}$ which immediately implies that
$$
\mu(\theta_n(E))=\mu(E) \quad \text{and} \quad \lambda(\theta_n(E)) = \lambda(E)
$$
for $n \gg 0$, where $\mu(E)$ and $\lambda(E)$ are the Iwasawa invariants of the $p$-adic $L$-function of $E$.  

For primes of good ordinary reduction, we have the following theorem.
\begin{theorem}
Let $E/\QQ$ be an elliptic curve with good ordinary reduction at $p$ such that $E[p]$ is irreducible as a Galois module.  If $\mu(E) = 0$, then 
$$
\mu(\theta_n(E)) = 0 \quad \text{and} \quad \lambda(\theta_n(E)) = \lambda(E) 
$$
for $n \gg 0$.  
\end{theorem}

\begin{proof}
See \cite[Proposition 3.7]{PW}.
\end{proof}

By contrast, for primes $p$ of good supersingular reduction, the $\lambda$-invariants of Mazur--Tate elements are always unbounded.  This is related to the fact that the $p$-adic $L$-function of $E$ is not an Iwasawa function and one instead has a pair of Iwasawa-invariants, $\mu^\pm(E)$ and $\lambda^\pm(E)$ as defined in \cite{pollack03} and \cite{sprung}.  In this case, results of Kurihara and Perrin-Riou imply that these invariants can be read off of the Iwasawa invariants of Mazur--Tate elements.

\begin{theorem}\label{thm:PW-ss}
Let $E/\QQ$ be an elliptic curve with good supersingular reduction at $p$.  
\begin{enumerate}
\item For $n \gg 0$,
$$
\mu(\theta_{2n}(E)) = \mu^+(E) \quad \text{and} \quad
\mu(\theta_{2n-1}(E)) = \mu^-(E).
$$
\item If $\mu^+(E) = \mu^-(E)$, then
$$
\lambda(\theta_n(E)) = q_n + \begin{cases} \lambda^+ & n \text{~even}\\
\lambda^- & n \text{~odd},
\end{cases}
$$
where 
$$
q_n = p^{n-1} - p^{n-2} + \dots + \begin{cases} p -1 & n \text{~even}\\
p^2 - p & n \text{~odd}.
\end{cases}
$$
\end{enumerate}
\end{theorem}

\begin{proof}
See \cite[Theorem 4.1]{PW}.
\end{proof}

\begin{remark}
    The $q_n$ term in the above formula forces the $\lambda$-invariants to be unbounded as $n$ grows.  The interpolation property of the Mazur--Tate elements then implies that the $p$-adic valuation of $L(E,\chi,1)/\Omega_E^+$ (where $\Omega_E^+$ is the real Néron period of $E$) is unbounded as $n$ increases.  The Birch and Swinnerton-Dyer conjecture thus predicts that some algebraic invariant should grow along the cyclotomic $\Zp$-extension.  Consistent with this, it is known that the Tate--Shafarevich group of $E$ (if finite) grows without bound along this extension (see \cite[Theorem 10.9]{kobayashi}).
\end{remark}

For elliptic curves over $\QQ$ with additive reduction at $p$, the Mazur--Tate elements do not immediately give rise to a $p$-adic $L$-function. Moreover, their $\lambda$-invariants are quite large as it was proven that $\lambda(\theta_n(E)) \geq p^{n-1}$ whenever $p$ is an additive prime for $E$ (see \cite[Corollary~5.3]{doyon-lei}).  Despite this, these $\lambda$-invariants appear to satisfy regular formulae as observed in \cite[\S~6]{doyon-lei}.  Indeed, in this work, it was observed that
$$
\lambda(\theta_n(E)) = ap^{n-1} + bq_{n-1} + O(1)
$$
in many examples where $p$ was an additive prime for $E/\QQ$ and $a$ and $b$ were constants independent of $n$.
A main objective of the present article is to give theoretical explanations of these growth patterns of $\lambda(\theta_n(E))$ at additive primes, and to relate them to other better understood invariants. 

\subsection{Main results}
We now discuss the main results we prove in the present article. For an elliptic curve $E/\QQ$ with additive reduction at a prime $p$, our approach differs depending on the `potential reduction' type of $E$. 

Recall that when $E$ has additive reduction at $p$, it achieves semistable reduction over a finite extension of $\QQ$. We first study the case where $E$ achieves semistable reduction over the quadratic field $F=\QQ(\sqrt{(-1)^{{(p-1)/2}}p})$ and relate the Mazur--Tate elements of $E$ with its quadratic twist associated with $F$, denoted by $E^{F}$. Since $E^F$ has good reduction at $p$, the Iwasawa invariants of the $p$-adic $L$-function(s) of $E^F$ are well understood. In particular, we prove:
\begin{lthm}[Theorem \ref{quad}]\label{thmA}
        Let $E/\QQ$ be an elliptic curve with additive reduction at an odd prime $p$. Let $i$ be an even integer between $0$ and $p-2$.
        Assume that 
        \begin{itemize}
        \item the quadratic twist $E^F$ has either good ordinary or multiplicative reduction at $p$; 
        \item the $\mu$-invariant of $L_p(E^F,\omega^{(p-1)/2+i}, T)$ is zero {and the $\mu$-invariant of $\theta_{n,i}(E)$ is non-negative when $n$ is sufficiently large.}
        \end{itemize}
Then, for all $n\gg0$,
    \begin{align*}
    \mu(\theta_{n,i}(E)) &= 0, \\ 
    \lambda(\theta_{n,i}(E))&= \frac{p-1}{2}\cdot{p^{n-1}} 
 + \lambda(E^F, \omega^{{(p-1)/2+i}})\end{align*}
where $\lambda(E^F, \omega^{{(p-1)/2+i}})$ denotes the $\lambda$ invariant of $L_p(E^F, \omega^{{(p-1)/2+i}}, T)$.
\end{lthm}
Our proof relies on a comparison of the interpolation properties of $\theta_{n,i}(E)$ with those of $\theta_{n,i+\frac{p-1}{2}}(E^F)$.  The corresponding interpolation formulae are nearly the same with the exception of the Néron periods. Here, the ratio of the Néron periods of $E$ and $E^F$ equals $\sqrt{p}$, up to a $p$-unit.  This factor of $\sqrt{p}$ leads to the presence of the term $\frac{p-1}{2}\cdot p^{n-1}$ in the formula above.

\begin{remark}
\label{rmk:periods}
The term $\frac{p-1}{2}\cdot p^{n-1}$ forces the $\lambda$-invariants to grow without bound.  However, unlike the good supersingular case, this is not explained via the Birch and Swinnerton-Dyer conjecture by the growth of the Tate--Shafaverich group along the cyclotomic $\ZZ_p$-extension.  Instead, it is explained by the growth of the $p$-valuation of the ratio of the periods $\Omega_{E/k_n}$ and $\left(\Omega_{E/\QQ}\right)^{p^n}$.  This ratio, in turn, captures the lack of a global minimal model for $E$ over the number field $k_n$. See \eqref{perratio} and Proposition~\ref{fudge}. 
\end{remark}

We can prove a similar result if $E^F$ has good supersingular reduction at $p$, where a formula of $\lambda(\theta_{n,i}(E))$ in terms of the plus and minus $p$-adic $L$-functions of $E^F$ is proven. The formula we prove resembles that of Theorem~\ref{thm:PW-ss}, except for the presence of the extra term $\frac{p-1}{2}\cdot p^{n-1}$ originating from the ratio of periods; see Theorem~\ref{ssquad} for the precise statement.

When $E$ has additive reduction at $p$, but achieves good ordinary reduction over more general extensions, we can derive exact formulae for the $\lambda$-invariants of Mazur--Tate elements, assuming the Birch and Swinnerton-Dyer conjecture. Specifically, we require the $p$-primary part of the Tate--Shafarevich group to be finite over $k_n$ and that the leading term of the Taylor expansion of $L(E/k_n,s)$ at $s=1$ predicted in the Birch and Swinnerton-Dyer conjecture holds up to $p$-adic units; see Conjecture~\ref{conj:pBSD}. In the following theorem, $\cX(E/\QQ_\infty)$ denotes the dual of the Selmer group of $E$ over $\QQ_\infty$. From the work of Delbourgo \cite{del-JNT}, we know that $\cX(E/\QQ_\infty)$ is a finitely generated torsion $\Lambda$-module when $p\geq 5$ is a prime of potentially good ordinary reduction. In particular, we can define the Iwasawa invariants $\mu(\cX(E/\QQ_\infty))$ and $\lambda(\cX(E/\QQ_\infty))$ are defined, and they are shown to be related to the Iwasawa invariants of $\theta_n(E)$.

\begin{lthm}[Theorem \ref{thm: bsd}]\label{thmB}
         Let $E/\QQ$ be an elliptic curve with additive, potentially good ordinary reduction at a prime $p\geq 5$ and minimal discriminant $\Delta_E$. 
        Assume that
        \begin{itemize}
            \item Conjecture~\ref{conj:pBSD} is true over $k_{n}$ for all $n \gg 0$,
            \item  $\lambda(\theta_{n}(E))<p^{n-1}(p-1)$ for $n\gg0$.
        \end{itemize}
Then, when $n$ is sufficiently large, we have 
        \begin{align*}
          \mu(\theta_{n}(E)) &= \mu(\cX(E/\QQ_\infty))\\
            \lambda(\theta_{n}(E))
 &= \frac{(p-1)\cdot \ord_p(\Delta_E)}{12}\cdot p^{n-1}+{\lambda(\cX(E/\QQ_\infty))}.
        \end{align*}
\end{lthm}
Our method is to analyze how each term in the Birch and Swinnerton-Dyer conjecture changes along the cyclotomic $\ZZ_p$-extension.  A key step here relies on a control theorem for the $p$-primary Selmer group of $E$ along $\QQ_\infty$, which in turn governs the growth of the Tate--Shafarevich groups (see Theorems~\ref{thm:control} and \ref{sha}).  From this analysis, we can determine the $p$-adic valuation of $L(E,\chi,1)/\Omega_E$ for Dirichlet characters $\chi$ of $p$-power conductor and thus the $\lambda$-invariant of $\theta_{n,0}(E)$.  The unbounded term in the above formula arises from terms that capture the lack of a global minimal model for $E$ over $k_n$. This formula is consistent with Theorem \ref{thmA}; when good ordinary reduction at $p$ is achieved over a quadratic extension, we have $\ord_p(\Delta_E)=6$.

With Theorems \ref{thmA} and \ref{thmB} in mind, we look for possible generalisations to Hecke eigenforms $f$ on $\Gamma_1(N)$ of general weight $\ge2$ with $a_p(f)=0$ and $p \mid N$. We formulate a ``potentially ordinary'' hypothesis for the $p$-adic Galois representation of $G_\QQ$ associated with $f$ (see \S\ref{sec: results on modularforms}) that is analogous to the property of potentially ordinary reduction for an elliptic curve. Building on the arguments of the recent work \cite{muller24} of Müller on CM modular forms, we are able to show if a cusp form (that is not necessarily CM) satisfies this hypothesis, then it gives rise a cusp form that is ordinary at $p$ after twisting by a character of conductor $p$. In particular, there exists a $p$-adic $L$-function for the twisted form which lies in $\Zp[[T]]$ (assuming the normalisation is given by the cohomological period). This allows us to deduce the following theorem:
\begin{lthm}[Theorem \ref{thm: modforms twist}]\label{thmC}
    Let $p$ be an odd prime and let $f \in S_k(\Gamma_1(N))$ be a $p$-new Hecke eigenform with $k\geq 2$, $p^2 \mid N$. 
    Let $V_f$ denote the $p$-adic representation of $G_\QQ$ associated with $f$. Assume furthermore that 
    \begin{itemize}
        \item  there exists a non-trivial unramified ${G_{\Qp(\mu_p)}}$-subrepresentation of $V_f|_{G_{\Qp(\mu_p)}}$ and 
        \item the Iwasawa invariants satisfy $\mu(\theta_{n,i}(f))=0$ and $\lambda(\theta_{n,i}(f))< p^{n-1}(p-1)$ for $n\gg 0$.
    \end{itemize}
     Then, there exist constants $m$ and $\lambda_0$ dependent only on $f$ and $i$, with $\frac{1}{p-1}\le m  \le \frac{p}{p-1}$ and $\lambda_0\ge0$, such that for all $n\gg 0$, 
    \[\lambda(\theta_{n,i}(f))= m\cdot(p-1)\cdot p^{n-1} + \lambda_0.\]
\end{lthm}

\begin{remark}Let $f \in S_k(\Gamma_0(N),\CC)$ be a Hecke eigenform with $k\geq 2$, $p\mid N$.
\begin{itemize}
    \item One can obtain an infinite family of examples of $f$ satisfying the first hypotheses of Theorem~\ref{thmC} by taking a $p$-ordinary form and twisting it by a character of conductor $p$ to obtain $f$. 
    \item In the formula above, $m = \ord_p(\Omega_f/\Omega_{\tilde{f}})$, where $\tilde{f}$ denotes the aforementioned twist of $f$ that is $p$-ordinary, and $\Omega_f$ (resp.\ $\Omega_{\tilde{f}}$) is  the cohomological period of $f$ (resp.\ $\tilde{f}$). 
\end{itemize}  
\end{remark}
We discuss possible further generalisations of Theorem \ref{thmB} in this direction. 

A control theorem for the Bloch--Kato Selmer groups associated with modular forms at potentially ordinary primes (which follows from techniques developed in \cite{OCHIAI200069}) can be used to establish growth formulae for the Bloch--Kato Tate--Shafarevich groups along the cyclotomic $\Zp$-extension over $\QQ$ assuming that the Pontryagin dual of the Bloch--Kato Selmer group of $f$ over $\QQ_\infty$ is a torsion $\Lambda$-module. This implies that the growth in $\lambda(\theta_{n, i}(f))$ exhibited by Theorem~\ref{thmC} cannot be attributed to the growth of the ($p$-primary) Bloch--Kato--Tate--Shafarevich group associated with $f$. Therefore, we expect the term $\lambda_0$ in Theorem \ref{thmC} to \emph{equal} the algebraic $\lambda$-invariant of the Bloch--Kato Selmer group attached to $f$, as is the case in Theorem \ref{thmB}.

To prove a formula as precise as the one given in Theorem \ref{thmB} in the context of general modular forms (i.e., a formula where the constant $m$ can be expressed in terms of arithmetic invariants attached to $f$), one would need to invoke the $p$-part of the Bloch--Kato Tamagawa number conjecture for the base-change of $V_f$ along the layers of the cyclotomic $\Zp$-extension of $\QQ$. We were unable to find a reference for such a formulation that is suitable for our purposes at the time of writing.




\subsection{Organisation}
We begin with preliminaries related to modular symbols and Mazur--Tate elements associated with elliptic curves over $\QQ$ in \S~\ref{sec:msmt}. In \S~\ref{sec:prelim}, we provide background on elliptic curves with additive reduction and review the notion of `potential semistability', i.e., when $E$ has additive reduction over a field $K$, but attains semistable reduction over a finite extension of $K$. Moreover, we study properties of the Selmer group associated with $E$ at additive potentially good ordinary primes. We use this to show that the growth of the $p$-primary part of the Tate--Shafarevich group of $E$ along the cyclotomic $\Zp$-extension of $\QQ$ is similar to the good ordinary case. In \S~\ref{sec:form1}, we prove Theorems~\ref{thmA} and \ref{thmB}. The potentially supersingular case in the generality of Theorem~\ref{thmB} has eluded us so far, but we provide examples and a conjecture supported by computational data from Sage in this setting in \S~\ref{ss: speculation for pot. ss.}. In this section, we also discuss a couple of approaches towards the potentially supersingular case. More computational data is provided in \S~\ref{pss data}. In \S\ref{sec: results on modularforms}, we prove Theorem~\ref{thmC} by showing the existence of a twist of $f$ which is ordinary at $p$. 

\subsection*{Acknowledgement}
The research of AL is supported by the NSERC Discovery Grants Program RGPIN-2026-07384. RP's research has been partially supported by NSF grant DMS-2302285 and by Simons Foundation Travel Support Grant for Mathematicians MPS-TSM-00002405.  Parts of this work were carried out during NP's summer internship at the University of Ottawa in the summer of 2023, supported by a MITACS Globalink Scholarship. This article forms part of the master's thesis of NP at IISER, Pune (available \href{http://dr.iiserpune.ac.in:8080/xmlui/handle/123456789/10021}{here}). The authors thank Anthony Doyon, Rylan Gajek-Leonard, Ariel Pacetti and Rik Sarkar for interesting discussions related to the content of the article. {We also thank the anonymous referees for a careful reading of an earlier version of this article, as well as for numerous comments and suggestions that led to many improvements.}
\section{Modular symbols and Mazur--Tate elements}\label{sec:msmt}
\subsection{Modular symbols}
Let $R$ be any commutative ring and, for any integer $g \geq 0$, let $V_g(R)$ be the space of homogeneous polynomials of degree $g$ in the variables $X$ and $Y$ with coefficients in $R$. Let $\dDelta$ denote the abelian group of divisors on $\mathbb{P}^1(\QQ)$, and let $\dDelta^0$ denote the subgroup of degree 0 divisors. Let $\SL_2(\ZZ)$ act on $\dDelta^0$, by linear fractional transformations, which allows us to endow $\Hom(\dDelta^0, V_{g}(R))$ with a right action of $\SL_2(\ZZ)$ via
$$(\varphi |_{\gamma})(D) = (\varphi(\gamma \cdot D))|_{\gamma},$$
where $\varphi \in \Hom(\dDelta^0, V_{g}(R))$, $\gamma \in \SL_2(\ZZ)$ and $D \in \dDelta^0$.
\begin{defn}\label{defn:modsymb}
    Let $\Gamma\leq \SL_2(\ZZ)$ be a congruence subgroup. We define $\Hom_{\Gamma}(\dDelta^0, V_g(R))$ to be the space of $R$-valued \textbf{modular symbols} of weight $g$, level $\Gamma$ for some commutative ring $R$, and we denote this space by $\Symb(\Gamma, V_g(R))$.
\end{defn}
\begin{remark}
    One can identify $\text{Symb}(\Gamma, {V_g(R)})$ with the compactly supported cohomology group $ H^1_c(\Gamma, {V_g(R)})$ (see \cite[Proposition~4.2]{ash-ste}).
\end{remark}
For $f \in S_k(\Gamma)$, we define the \textbf{modular symbol associated with $f$} as 
\[\xi_f: \{s\}-\{r\} \to 2\pi i \int_s^r f(z)(zX+Y)^{k-2}dz,\] 
which is an element of $\Symb(\Gamma, V_{k-2}(\CC))$ as $f$ is a holomorphic cusp form. Let $A_f$ be the field of Fourier coefficients of $f$ and fix a prime $p$. The matrix $\iota \colonequals \begin{psmallmatrix}
    -1& 0 \\ 0 & 1
\end{psmallmatrix}$ acts as an involution on $\Symb(\Gamma, \CC)$ and we decompose $\xi_f=\xi_f^+ + \xi_f^-$ with $\xi_f^\pm$ in the $\pm1$-eigenspace of $\iota$ respectively. By a theorem of Shimura, there exist $\Omega_f^\pm \in \CC$ such that ${\xi_f^\pm/\Omega_f^\pm}$ take values in $V_{k-2}(A_f)$, and in $V_{k-2}(\overline{\QQ}_p)$ upon fixing an embedding of $\overline{\QQ}\hookrightarrow \overline{\QQ}_p$ (which we fix for the rest of the article). Define $\Psi_f^\pm \colonequals \psi_f^\pm/\Omega_f^\pm$, and $\Psi_f \colonequals \Psi_f^+ + \Psi_f^-$ which is in $\Symb(\Gamma, \overline{\QQ}_p)$.  

\begin{remark}[\textbf{On periods}]\label{rem:periods}
Let $\mathcal{O}_f$ denote the ring of integers of the completion of the image of $A_f$ in $\overline{\QQ}_p$. We can choose $\Omega^+$ and $\Omega^-$ so that each of $\Psi_f^+$ and $\Psi_f^-$ takes values in $V_{k-2}(\mathcal{O}_f)$ and that each takes on one value having at least one coefficient in $\mathcal{O}_f^\times$. We denote these periods by $\Omega_f^\pm$; they are called \textbf{cohomological periods} of $f$, which are well-defined up to $p$-adic units (for more details, see \cite[Def. 2.1]{PW}).

For an elliptic curve $E$, we are supplied with the (real and imaginary) \textbf{Néron periods}, which we denote by $\Omega_E^\pm$. They ensure that the modular symbols take values in $\Qp$ but \textit{a priori} do not guarantee integrality. In \S\ref{sec:form1}, we will implicitly assume that the $p$-adic $L$-function of an elliptic curve $E$ is constructed using the Néron periods of $E$. In \S\ref{sec: results on modularforms}, we work with higher weight modular forms, so cohomological periods will be in use.
\end{remark}
\subsection{Mazur--Tate elements and $p$-adic $L$-functions}\label{ssec: MT and Lp}
Recall the following notation given in the introduction. We fix an elliptic curve $E/\QQ$ and let $f_E$ be the weight 2 newform associated with $E$ by the modularity theorem. For a non-negative integer $n$, let $\mathcal{G}_n \colonequals \Gal(\QQ(\mu_{p^n})/\QQ)$. 
For $a \in (\ZZ/p^n\ZZ)^\times$, we write $\sigma_a\in\cG_n$ for the element that satisfies $\sigma_a(\zeta)=\zeta^a$ for $\zeta \in \mu_{p^n}$.
\begin{defn}
    For a modular symbol $\varphi \in \Symb(\Gamma, V_g(R))$, define the associated Mazur--Tate element of level $n\geq 1$ by 
\[\vartheta_n(\varphi)= \sum_{a \in (\ZZ/p^n\ZZ)^\times}\varphi(\{\infty\}-\{a/p^n\})|_{(X,Y)=(0,1)}\cdot \sigma_a \in R[\mathcal{G}_n].\]
    When $R$ is a subring of $\overline{\QQ}_p$, decomposing $\mathcal{G}_{n+1}=G_n\times(\ZZ/p\ZZ)^\times$ with $G_n\cong\Gal(k_{n}/\QQ)$, one can project $\vartheta_n(\varphi)$ to $R[G_n]$ by the characters $\omega^i: (\ZZ/p\ZZ)^\times \to \Zp^\times$, where $0\leq i \leq p-2$. We define the \emph{$\omega^i$-isotypic component of the $p$-adic Mazur--Tate element} of level $n$ associated with a
    cusp form $f\in S_k(\Gamma)$ as 
    \[\theta_{n,i}(f)\colonequals \omega^i(\vartheta_{n+1}(\Psi_f)) \in \overline{\QQ}_p[G_n].\] 
\end{defn} 
We define $\theta_{n,i}(E)\colonequals\theta_{n,i}(\Psi_{f_E}) \in \Qp[G_n]$, where the normalisation is given by either of the two sets of periods discussed in Remark \ref{rem:periods}. 
\begin{proposition}\label{interpprop}
    For a character $\chi$ on $G_n$, $\theta_{n, i}(f)$ satisfies  the following interpolation property 
\[\chi(\theta_{n,i}(f))=\tau(\omega^i\chi)\cdot\frac{L(f, \overline{\omega^i\chi},1)}{\Omega^{\epsilon}},\]
where $\tau$ denotes the Gauss sum, and $\epsilon\in\{+,-\}$ is the sign of $\omega^i(-1)$.
\end{proposition}
\begin{proof}
    See \cite[Equation 8.6]{MTT}, and consider the projection described above. 
\end{proof}
Let $\gamma_n$ be a generator of ${G}_n$. Then, for any element $F \in \Zp[{G}_n]$, we may write it as a polynomial 
$\sum_{i=0}^{p^n-1}a_iT^i$ with $T=\gamma_n-1$. 
\begin{defn}[Iwasawa invariants]
    The $\mu$ and $\lambda$-invariants of $F=\sum_{i=0}^{p^n-1}a_iT^i \in \Zp[G_n]$ are defined as 
\begin{align*}
    \mu(F) &= \underset{i}{\min}\{\ord_p(a_i)\},\\
    \lambda(F) &= \min\{ i : \ord_p(a_i) = \mu(F)\}
\end{align*}
where $\ord_p$ is the $p$-adic valuation such that $\ord_p(p)=1$. 
\end{defn}
These invariants are independent of the choice of $\gamma_n$. One can directly define $\mu$ and $\lambda$-invariants for an element of the finite level group algebra $\Zp[G_n]$ which are equivalent to the above definitions; for more details, see \cite[\S~3.1]{PW}.

Finally, we briefly recall the construction of the $p$-adic $L$-function of a modular form $f \in S_k(\Gamma_0(N))$ when $p \nmid a_p(f)\cdot N$ in terms of its Mazur--Tate elements. See \cite[\S~3.2]{PW} for more details.

Let $\alpha$ denote the unique $p$-adic unit root of the Hecke polynomial $X^2-a_p(f)X+p^{k-1}$. We consider the $p$-stabilisation \[f_{\alpha}(z)\colonequals f(z)- \frac{p^{k-1}}{\alpha}f(pz),\]
which results in a norm-compatible system given by $\{\frac{1}{\alpha^{n+1}} \theta_{n,i}(f_{\alpha})\}_n$.
Then, \[L_p(f, \omega^i)=\varprojlim_{n}\frac{1}{\alpha^{n+1}} \theta_{n,i}(f_{\alpha})\]
is the $\omega^i$-isotypic component of the $p$-adic $L$-function attached to $f$. This is an element of $\Lambda$ if we normalise the modular symbols by the cohomological periods. When $f$ corresponds to an elliptic curve over $\QQ$, we use the N\'eron periods for normalisation, so $L_p(E,\omega^i) \in \Lambda \otimes \Qp$. We use the notation $L_p(f, \omega^i, T)$ for the image of $L_p(f, \omega^i)$ under the isomorphism $\Lambda\cong\Zp[[T]]$. 
\section{Preliminaries: Elliptic curves and additive reduction}\label{sec:prelim}
In this section, we recall certain facts about elliptic curves over number fields that have additive reduction at a finite place $v$ above $p$. We shall consider the base-change of an elliptic curve $E/\QQ$ to a number field, as well as the completion of a number field at a finite place (to which we refer as a $p$-adic field). We say that $E$ has \textit{semi-stable} reduction at $v$ if it has either good or multiplicative reduction at $v$. We begin with the following well-known result.
\begin{theorem}[Semi-stable reduction theorem]\label{thm:semistable}
    Let $K$ be a $p$-adic field. There exists a finite extension $K'/K$ such that $E$ has semi-stable reduction over $K'$. 
\end{theorem}
\begin{proof}
See \cite[Proposition VII.5.4]{Si}.
\end{proof}
\begin{remark}
   We recall that if $E$ has additive reduction at $p$, it attains semi-stable reduction at a place $v$ after a base change to a finite extension. If it has good reduction at $p$, then the reduction type remains the same for any places above $p$. If it has nonsplit multiplicative reduction at $p$, it becomes split after a base change to a quadratic extension.
\end{remark}
We say that $E$ has \textit{potentially good reduction} at $p$ if there exists a finite extension $F/\QQ$ such that the base-change of the curve to $F$ has good reduction at the places of $F$ above $p$. By \cite[ Prop. VII.5.5]{Si}, this is equivalent to saying that the $j$-invariant of the curve is a $p$-adic integer.
\textit{Potentially multiplicative reduction} is defined in a similar way.
\subsection{Potentially good reduction}\label{ssec: potgoodred}
In this subsection, we assume that $E$ has potentially good reduction at $p$. Let $K$ be a $p$-adic field. Let $m$ be an integer greater than 2 and coprime to $p$. Let $K^{ur}$ be the maximal unramified extension of $K$. Define $L\colonequals K^{ur}(E[m])$. The extension $L$ is independent of $m$. Moreover, we have the following lemma. 
\begin{lemma}[Serre--Tate]
The field $L$ is the minimal extension of $K^{ur}$ where $E$ achieves good reduction.
\end{lemma}
\begin{proof}
    See \cite[Section 2, Corollaries 2 and 3]{serretate}.
\end{proof}
Write $\Phi\colonequals \Gal(L/K^{ur})$ and define the \emph{semistability defect} of $E$ as $e\colonequals \#\Phi$ (which depends on $E$ and $p$; we suppress it from the notation for simplicity). We see that $\Phi$ is the inertial subgroup of $\Gal(L/K)$. For a description of $\Phi$ in the case when $p\in\{2,3\}$, see \cite{Kraus1990}. 

When $p\ge5$, the discussion in \cite[Section 5.6]{Serre1971/72} tells us that $\Phi$ is cyclic of order 2, 3, 4 or 6.  Furthermore, the size of $\Phi$ is given by
 \begin{equation}\label{eq: semistabilitydef}
 e = \frac{12}{\text{gcd}(12,\ord_p(\Delta_E))},
 \end{equation}
 where $\Delta_E$ is the minimal discriminant of $E/\QQ$.
 This allows us to show, for $p\geq 5$, that $E$ achieves good reduction over an extension of degree at most $6$.
\begin{lemma}\label{lem: Kgdeg}
    Let $p\geq 5$. Suppose that $E$ has additive potentially good reduction at $p$. Then the semistability defect $e$ is the smallest integer $e\in \{2,3,4,6\}$ such that $E$ obtains good reduction over $\Qp(\sqrt[e]{p})$.
\end{lemma}
\begin{proof}
    In this case, $\Phi= \Gal(L/\Qp^{ur})$ is cyclic of order $e$. So, $L/\Qp^{ur}$ is tamely ramified and cyclic of order $e$, thus $L=\Qp^{ur}(\sqrt[e]{p})$. As good reduction is invariant under unramified extensions, we conclude that $E$ attains good reduction over $\Qp(\sqrt[e]{p})$. 
\end{proof}
\begin{lemma}\label{ediv}
     Assume that $E$ has potentially good reduction at $p\geq 5$ and
that $e>2$. Then $E$ is potentially ordinary at $p$ if and only if $e$ divides $p-1$. If $E$ is potentially supersingular at $p$ then $e$ divides $p+1$.
\end{lemma}
\begin{proof}
    See \cite[Lemma 2.1]{del-JNT}.
\end{proof}
\subsection{Potentially multiplicative reduction}\label{sec:potmult}
When $E/\QQ$ has potentially multiplicative reduction at $p$, it achieves multiplicative reduction over a quadratic extension. This is because the $j$-invariant of $E$ has negative $p$-adic valuation, and thus $E$ becomes isomorphic to a \emph{Tate curve} upon taking a base change to a quadratic extension by \cite[Theorem 5.3, Corollary 5.4]{silverman1994advanced}. See also \cite[Section 5.6 (b)]{Serre1971/72}. 
\subsection{The Birch--Swinnerton-Dyer conjecture over number fields}\label{ssec: BSD}
The Birch and Swinnerton-Dyer conjecture for elliptic curves over a number field $K$ provides an expression for the leading term of the $L$-function $L(E/K, s)$ at $s=1$ in terms of arithmetic data of $E/K$, which we recall below.
\begin{conjecture}\label{conj:BSD}
    Let $K$ be a number field. Then
\begin{itemize}
    \item $\ord_{s=1} L(E/K,s) = \textup{rank}(E/K)$, 
    \item the Tate--Shafarevich group of $E/K$, denoted by $\Sha(E/K)$ is finite and
    \item the leading term of the Taylor series at $s\!=\!1$ of the $L$-function $L(E/K, s)$ is given by
\[
 \frac{L^{(r)}(E/K,s)}{\Omega_{E/K}}=\frac{\textup{Reg}({E/K})|\Sha{(E/K)}| \mathrm{Tam}({E/K})}{\sqrt{|\Delta_K|}|E(K)_{\textup{tors}}|^2},
 \tag{$\dagger$}\label{bsd1}
\]
\end{itemize}
where $r$ is the order of vanishing of $L(E/K, s)$ at $s=1$, $\Delta_K$ is the discriminant of $K$, $\textup{Reg}$ denotes the regulator and $\mathrm{Tam}({E/K})$ is the product of Tamagawa numbers at finite places.
\vspace{3pt}\\
Here, $\Omega_{E/K} \in \CC^\times$ is a `period' of $E$ which has a precise description in terms of differentials on $E(K)$ and its completions (see Definition~\ref{defn: period} below). We will refer to the expression on the right-hand side of \eqref{bsd1} as $\textup{BSD}(E/K)$. 
\end{conjecture}
For our purposes, we will utilize the "$p$-part" of Conjecture~\ref{conj:BSD}.

\begin{conjecture}\label{conj:pBSD}
    Let $K$ be a number field. Then
\begin{itemize}
    \item $\ord_{s=1} L(E/K,s) = \textup{rank}(E/K)$, 
    \item the $p$-primary part of the Tate--Shafarevich group, $\Sha(E/K)[p^\infty]$, is finite and
    \item the leading term of the Taylor series at $s\!=\!1$ of the $L$-function $L(E/K, s)$ satisfies
\[
 \ord_p\left(\frac{L^{(r)}(E/K,s)}{\Omega_{E/K}}\right)=\ord_p\left(\frac{\textup{Reg}({E/K})|\Sha{(E/K)[p^\infty]}| \mathrm{Tam}({E/K})}{\sqrt{|\Delta_K|}|E(K)_{\textup{tors}}|^2}\right),
 \tag{$\dagger_p$}\label{bsdp}
\]
\end{itemize}
where the notation is the same as Conjecture \ref{conj:BSD}.
\end{conjecture}

\subsubsection{Periods in the Birch and Swinnerton-Dyer conjecture}
Let $K$ be a number field. Let $v$ be a non-archimedean place of $K$ and write $K_v$ for the completion of $K$ at $v$ with ring of integers $\mathcal{O}_v$, and choose a uniformizer $\pi_{K_v}$. Let $q_v$ be the cardinality of the residue field. Let $|\cdot|_v$ denote the unique normalized absolute value on $K_v$ with $|\pi_{K_v}|_v=\frac{1}{q_v}$.

Given an elliptic curve $E$ defined over $K$ (for our purposes, it is the base-change of $E/\QQ$), for each non-archimedean place $v$ of $K$, we can find a \emph{minimal} Weierstrass equation for $E$. Consequently, there is an associated discriminant $\Delta_v$ and a (minimal) invariant differential $\omega_v^{\min}$. When the class number of $K$ is 1, there exists a global minimal Weierstrass equation (i.e., minimal for the base-change of $E$ to $K_v$ for all non-archimedean places $v$ of $K$); see \cite[\S VIII.8]{Si}. This does not hold for general number fields. We discuss the factor in Conjecture \ref{conj:BSD} that encapsulates this phenomenon. 

The set of local points $E(K_v)$ admits a structure of a $K_v$-analytic manifold of dimension 1. For an open subset $U\subset E(K_v)$, an open subset $V \subset K_v$ and a chart $\beta:U \to V$, $\omega_v^{\min}$ is of the form $f(z)dz$ on $V$, where $dz$ is the usual differential on $K$ and $f$ is a Laurent power series in $z$ without poles in $V$. We define 
    \[\int_{U}|\omega_v^{\min}|_v := \int_V |f(z)|_v d\mu,\]
    where $\mu$ is the Haar measure on $K_v$ normalized so that $\mathcal{O}_v$ has volume $1$. The integral over $E(K_v)$ is defined by gluing these charts. The following relates the Tamagawa number with the integral over $E(K_v)$.   
\begin{lemma}
Denote the \emph{Tamagawa number} at $v$ by $c(E/K_v)$. We have
    \[\int_{E(K_v)}|\omega_v^{\min}|_v= c(E/K_v)\cdot{L_v(E, q_v^{-1})}.\] 
\end{lemma}
\begin{proof}
    See \cite[Lemma 1.5]{AdamMorgan}. 
\end{proof}
If $\omega$ is a non-zero global differential on $E$, there exists $\lambda \in K_v$ such that  $\omega= \lambda \omega_v^{\min}$ and 
\[\int_{E(K_v)}|\omega|=|\lambda|_v\frac{c(E/K_v)|\tilde{E}_{ns}(k)|}{q}= \left|\frac{\omega}{\omega_v^{\min}}\right|_v c(E/K_v)\cdot L_v(E, q_v^{-1}).\]

We now give the definition of the period occurring in \eqref{bsd1}.
\begin{defn}\label{defn: period}
    For a global differential $\omega$ for $E$ over a number field $K$, we define
\begin{align*}
    \Omega_{E/\CC, \omega}&\colonequals2\int_{E(\CC)}\omega \wedge \overline{\omega},\\
    \Omega_{E/\mathbb{R}}&\colonequals\int_{E(\mathbb{R})}|\omega|,\\
    \Omega^{*}_{E/\mathbb{R}}&\colonequals\frac{\Omega_{E/\CC, \omega}}{\Omega_{E/\mathbb{R}, \omega}^2}.
\end{align*}
We define the \textbf{global period}
\[\Omega_{E/K}=\prod_{v\nmid\infty}\left|\frac{\omega}{\omega_v^{\min}}\right|_v\cdot\prod_{v \mid \infty}\Omega_{E/K_v, \omega}.\]
\end{defn}

\begin{remark}
    For $K=\QQ$, the global minimal differential $\omega$ is also $\omega_v^{\min}$ for all primes $v$. Thus,  
    \[\Omega_{E/\QQ}=\Omega_{E/\mathbb{R}},\]
    which is the usual (real) Néron period for $E$.
\end{remark}
\begin{lemma}\label{dok}
    Let $E$ be an elliptic curve defined over a number field $K$. Let $F/K$ be a finite extension. Then 
    \[\Omega_{E/F}= \Omega_{E/K}^{[F:K]}\prod_{v  \textup{ real}}(\Omega^*_{A/K_v})^{\#\{w\mid v \textup{ complex}\}}\prod_{v, w\mid v} \left|\frac{\omega_v^{\min}}{\omega_w^{\min}}\right|_{w},\]
    where $v$ runs over places of $K$ and $w$ over places of $F$ above $v$.
\end{lemma}
\begin{proof}
    This is \cite[Lemma 2.4]{Dokchitser_Dokchitser_2015}. 
\end{proof}
We see that for $F=k_n$ (which is a totally real field) and $K=\QQ$, we have 
\begin{equation}\label{perratio}
    \Omega_{E/k_n}= \Omega_{E/\QQ}^{p^n} \prod_{v, w\mid v} \left|\frac{\omega_v^{\min}}{\omega_w^{\min}}\right|_{w},
\end{equation}
where $v$ runs over all places of $\QQ$ and $w$ over places of $k_n$ above $v$. 
We conclude with the following explicit description of the terms appearing in the product in \eqref{perratio}.
\begin{proposition}\label{fudge}
Let $E/K$ be an elliptic curve over a number field, $F/K$ a field extension of
finite degree $d$. Let $v$ be a finite place of $K$ with $w\mid v$ a place of $F$ lying above above it. Let $\omega_v^{\min}$ and $\omega_w^{\min}$ be the minimal differentials for $E/K_v$ and $E/F_w$, respectively.
    \begin{enumerate}
        \item If $E/K_v$ has good or multiplicative reduction, then $\displaystyle\left|\frac{\omega_v^{\min}}{\omega_w^{\min}}\right|_{w}=1$.
        \item  If $E/K_v$ has potentially good reduction and the residue characteristic is not $2$ or $3$, then $\displaystyle\left|\frac{\omega_v^{\min}}{\omega_w^{\min}}\right|_{w}= q^{\left\lfloor e_{F/K} \ord_v(\Delta_{\min, v})/12\right\rfloor}$, where $q$ is the size of the residue field at $w$, and $e_{F/K}$ is the ramification index of $F_w/K_v$ .
    \end{enumerate}
\end{proposition}
\begin{proof}
    This is proved in \cite[Lemma 36 (5), (6)]{DokchitserEvansWiersema+2021+199+230}.
\end{proof}
\subsection{Iwasawa theory at potentially good, ordinary primes}
In this subsection, $K$ denotes a number field. Let $\overline{K}$ be an algebraic closure of $K$ and for any place $v$ of $K$, let $K_v$ denote the completion at $v$. Let $H^1(K, A)$ denote the cohomology group $H^1(\Gal(\overline{K}/K),A)$ for any $\Gal(\overline{K}/K)$-modules $A$. Similarly, let $H^1(L/K, A)$ denote $H^1(\Gal(L/K),A)$.
We define the $n$-Selmer group of $E/K$ as
\[\Sel_n(E/K) \colonequals \text{ker}\left(H^1(K, E[n])\to  \prod_v \frac{H^1(K_v, E[n])}{\text{im}(\kappa_v)}\right),\]
where $\kappa_v:E(K_v)/nE(K_v) \to H^1(K_v, E[n])$ is the Kummer map. Let
\[\mathcal{G}_E(K) \colonequals \text{im}\left(H^1(K,E[n]) \to \prod_v \frac{H^1(K_v, E[n])}{\text{im}(\kappa_v)}\right)\] where $v$ runs over all places of $K$. We have the following exact sequence 
\[0 \xrightarrow{} \text{Sel}_n(E/K) \xrightarrow{} H^1(K,E[n]) \xrightarrow{} {\mathcal{G}_E(K)} \xrightarrow{} 0.  \]
We begin with a lemma regarding Selmer groups over finite Galois extensions.   
\begin{lemma}\label{lem: sel1}
  Let $F/K$ be a finite Galois extension of degree $d$ such that $(n,d)=1$. Then 
  \[\Sel_n(E/K) \cong \Sel_n(E/F)^{\Gal(F/K)}.\]
\end{lemma}
\begin{proof}
    Let $G := \Gal(F/K)$. The inflation-restriction exact sequence gives:
    \[0\to H^1(F/K, E(F)[n])\to H^1(K, E[n]) \to H^1(F, E[n])^G \to H^2(F/K, E(F)[n]).\]  
    The first and last terms of this exact sequence are finite groups that are annihilated by both $n$ and $d$. As $n$ and $d$ are coprime, both groups are trivial. Thus, the restriction map $\res: H^1(K, E[n]) \to H^1(F, E[n])^G$ is an isomorphism.
    
    We have the following commutative diagram with exact rows.
\[\begin{tikzcd}
	0 & {\text{Sel}_n(E/K)} && {H^1(K,E[n])} && {\mathcal{G}_E(K)} & 0 \\
	\\
	0 & {\text{Sel}_n(E/F)^G} && {H^1(F, E[n])^G} && {\mathcal{G}_E(F)^G}
	\arrow[from=1-1, to=1-2]
	\arrow[from=1-2, to=1-4]
	\arrow["s", from=1-2, to=3-2]
	\arrow[from=1-4, to=1-6]
	\arrow["\res", from=1-4, to=3-4]
	\arrow[from=1-6, to=1-7]
	\arrow["g", from=1-6, to=3-6]
	\arrow[from=3-1, to=3-2]
	\arrow[from=3-2, to=3-4]
	\arrow[from=3-4, to=3-6]
\end{tikzcd}\]
As $\res$ is an isomorphism, the snake lemma gives the following exact sequence:
\[0 \to \text{ker}(s) \to 0 \to \text{ker}(g) \to \text{coker}(s) \to 0.\]
We show that $\text{ker}(g)=0$ below.

For a prime $v$ of $K$, let $w\mid v$ be a prime of $F$ and consider the natural restriction map $r_v: {H^1(K_v, E[n])}/{\text{im}(\kappa_v)} \to {H^1(F_w, E[n])}/{\text{im}(\kappa_w)}$. Then $\text{ker}(g)= \mathcal{G}_E(K) \cap \text{ker}(\prod_v r_v)$, so it suffices to show $\text{ker}(r_v)=0$ for all $v$. The exact sequence
\[0 \to E(K_v)/nE(K_v) \to H^1(K_v, E[n]) \to H^1(K_v, E(\overline{K_v}))[n]\to 0 ,\]
implies that
\[\frac{H^1(K_v, E[n])}{\text{im}(\kappa_v)} \cong H^1(K_v, E(\overline{K_v}))[n].\]
Similarly, we have 
\[\frac{H^1(F_w, E[n])}{\text{im}(\kappa_w)} \cong H^1(F_w, E(\overline{F_w}))[n].\]
Thus, it suffices to show that the restriction map $r_{w,v}:H^1(K_v, E(\overline{K_v}))[n] \to H^1(F_w, E(\overline{F_w}))[n]$ is injective. As $\ker(r_{w,v})=H^1(F_w/K_v, E(F_w))[n]$, which is annihilated by $[F_w:K_v]$ and $n$, it follows that $\text{ker}(r_{w,v})=0$, as desired.
\end{proof}
We define the $p$-primary Selmer group
\[\text{Sel}_{p^\infty}(E/K) = \lim_{\longrightarrow}\text{Sel}_{p^k}(E/K).\]
For a finite Galois extension $F/K$ with degree co-prime to $p$, Lemma~\ref{lem: sel1} implies that
\[\text{Sel}_{p^\infty}(E/K)\cong \text{Sel}_{p^\infty}(E/F)^{\Gal(F/K)}.\]

For $E/\QQ$ with additive potentially good reduction at a prime $p$, we establish Mazur's control theorem for $p^\infty$-Selmer groups of $E$ along the $\Zp$-extension of $\QQ$.
\begin{theorem}\label{thm:control}
    Let $E/\QQ$ be an elliptic curve with additive potentially good ordinary reduction at $p\geq 5$. Then Mazur's control theorem holds for ${\Sel}_{p^\infty}(E/\QQ_\infty)$, i.e., the kernel and the cokernel of the restriction map
    \[{\Sel}_{p^\infty}(E/k_n) \to {\Sel}_{p^\infty}(E/\QQ_\infty)^{\Gamma_n}\]  are finite. Furthermore, their cardinalities are bounded independently of $n$.
\end{theorem}
\begin{proof}
Let $K$ denote the minimal {Galois} extension of $\QQ$ over which $E$ achieves good reduction{. Recall from Lemma~\ref{lem: Kgdeg} that $K\subseteq \QQ(\sqrt[e]{p},\mu_e)$, where $e\in\{2,3,4,6\}$). In particular, $[K:\QQ]$ is a divisor of $e\times\phi(e)$, where $\phi$ is Euler's totient function}. Let $K_\infty\colonequals K\QQ_\infty$. We have $\Gal(K_\infty/K)\cong \Gamma$. Denote $\Gal(K/\QQ)$ by $G$. Then, for $p\geq 5$, we have $(|G|, p) = 1$. {As $\QQ_\infty/\QQ$ is a pro-$p$ extension, we have $K\bigcap\QQ_\infty=\QQ$,} If we write $K_n=(K_\infty)^{\Gamma_n}$, we have
\[G \cong \Gal(K_n/k_n)  \cong \Gal(K_\infty/\QQ_\infty),\quad {n\ge0}.\]

Lemma \ref{lem: sel1} gives
\begin{equation}
    \label{eq:Sel-iso1}{\Sel}_{p^\infty}(E/\QQ_\infty)\cong \Sel_{p^\infty}(E/K_\infty)^G,
\end{equation}
and 
\begin{equation}\label{eq:Sel-iso2}
    \text{Sel}_{p^\infty}(E/k_n)\cong \text{Sel}_{p^\infty}(E/K_n)^G
\end{equation}
when $n$ is sufficiently large.
As $E$ has good ordinary reduction at the primes of $K$ lying above $p$, Mazur's control theorem along the $\Zp$-extension $K_\infty/K$ {(see \cite[Proposition 6.4]{Mazur1972} and \cite[Theorem 1.2]{greenberg})} tells us that the kernel and cokernel of the restriction map
\[r'_{n}: \text{Sel}_{p^\infty}(E/K_n) \to \text{Sel}_{p^\infty}(E/K_\infty)^{\Gamma_n}\]
are finite and bounded independently of $n$.

The restriction map $r_n:\Sel_{p^\infty}(E/k_n)\rightarrow\Sel_{p^\infty}(E/\QQ_\infty)^{\Gamma_n}$ fits into the following commutative diagram:
\[
\begin{CD}
\Sel_{p^\infty}(E/k_n) @>{r_n}>> \Sel_{p^\infty}(E/\QQ_\infty)^{\Gamma_n}\\
@V{}VV       @VV{}V \\
\Sel_{p^\infty}(E/K_n)@>{r_n'}>> \Sel_{p^\infty}(E/K_\infty)^{\Gamma_n},
\end{CD}
\]
where the vertical maps are induced by restrictions. Note that $$(\Sel_{p^\infty}(E/K_\infty)^G)^{\Gamma_n}=(\Sel_{p^\infty}(E/K_\infty)^{\Gamma_n})^{G}.$$
Combined with \eqref{eq:Sel-iso1} and \eqref{eq:Sel-iso2}, we deduce the following commutative diagram
\[
\begin{CD}
\Sel_{p^\infty}(E/k_n) @>{r_n}>> \Sel_{p^\infty}(E/\QQ_\infty)^{\Gamma_n}\\
@V{\mathrel{\rotatebox{90}{$\sim$}}}VV       @VV{\mathrel{\rotatebox{90}{$\sim$}}}V \\
\Sel_{p^\infty}(E/K_n)^G @>{r_n'}>> \left(\Sel_{p^\infty}(E/K_\infty)^{\Gamma_n}\right)^G,
\end{CD}
\]
which tells us that
$\ker (r_n)= \ker (r_{n}')^G$ and $\mathrm{im} (r_n)=\mathrm{im} (r_n')^G$. 
Furthermore, as the order of $G$ is coprime to $p$ and $\mathrm{im}(r_{n}')$ is a {pro-$p$} group, we have $H^1(G,\mathrm{im}(r_{n}'))=0$. Taking $G$-cohomology of the short exact sequence
\[
0\rightarrow\mathrm{im}(r_{n}')\rightarrow \Sel(E/K_\infty)^{\Gamma_n}\rightarrow\coker(r_{n}')\rightarrow0
\]
gives $\coker(r_{n}')^G=\coker(r_n)$. {As $\ker(r_n')$ and $\coker(r_n')$ are finite and bounded independently of $n$, the same holds for $\ker(r_n)$ and $\coker(r_n)$, as desired.}
\end{proof}
Define the Pontryagin dual of $\Sel_{p^{\infty}}(E/\QQ_\infty)$ as
\[\cX(E/\QQ_\infty) \colonequals \textup{Hom}(\text{Sel}_{p^\infty}(E/\QQ_\infty), \QQ_p/\ZZ_p).\]
We define $\cX(E/K_\infty)$ similarly, where $K$ is as defined in the proof of Theorem~\ref{thm:control}. The following is a result of Delbourgo from \cite[Theorem (A)]{del-JNT}. 
\begin{theorem}\label{thm: delselmertorsion}
    Let $E$ be an elliptic curve defined over $\QQ$ with potentially good ordinary reduction at a prime $p\geq 5$. Then, $\cX(E/\QQ_\infty)$ is a torsion $\Lambda$-module.
\end{theorem}

The conclusion of Theorem~\ref{thm: delselmertorsion}, combined with the control theorem given in Theorem~\ref{thm:control}, implies that $\rank(E(k_n))$ is bounded above by the $\lambda$-invariant of $\cX(E/\QQ_\infty)$. Let $r_\infty=\displaystyle\lim_{n\rightarrow\infty}\rank(E(k_n))$. We have:
\begin{theorem}\label{sha}
    Assume that $E$ is an elliptic curve defined over $\QQ$ and that $E$ has potentially good ordinary reduction at $p \geq 5$. Furthermore, assume that $\Sha(E/\QQ_n)[p^\infty]$ is finite for all $n$. Then there exist integers $\lambda_E, \mu\geq 0$ and $\nu$ depending only on $E$ such that 
    \[|\Sha_E(\QQ_n)[p^\infty]|=p^{(\lambda_E- r_\infty)n + \mu p^n + \nu} \text{ for all } n\gg0.\]
\end{theorem}
\begin{proof}
    The argument for the good ordinary case as given in \cite[proof of Theorem~1.10]{greenberg} carries over under our hypotheses. {Indeed, the argument provided on p.~82 of \textit{op.\ cit.}\ is based solely on general properties of torsion $\Lambda$-modules together with the control theorem. Since we have established an analogue of the control theorem (Theorem~\ref{thm:control}), the same reasoning applies verbatim in our setting.}
\end{proof}
\section{Formulae for $\lambda$-invariants at additive primes}\label{sec:form1}
\subsection{Potential semi-stable reduction over a quadratic extension}
We first focus on the case where  $E/\QQ$ is additive at $p$ and achieves good or multiplicative reduction over a quadratic extension, i.e., the case when the semistability defect $e$ is equal to $2$. Let $E^F$ be the quadratic twist of $E$ over $F\colonequals\QQ(\sqrt{(-1)^{{(p-1)/2}}p})$ as in \S~\ref{sec:intro}. We begin with the following proposition that can be obtained from an analysis of the discriminant, and the invariants $c_4$ and $c_6$ associated with the minimal Weierstrass equations for $E$ and $E^F$, respectively. 
\begin{proposition}
    Let $E$ be an elliptic curve defined over $\QQ$ with additive reduction at $p$ such that $e=2$. Then $E^F$ has semistable reduction at $p$. 
\end{proposition}
\begin{proof}
    See \cite[Remark 2.3 and Proposition 2.4]{pal}.
\end{proof}
Next, we recall the main theorem of \cite{pal}, which gives a relation between the Néron periods of $E$ and those of its quadratic twist, applied to the additive case.
\begin{theorem}\label{thm: pal}
    Let $E^F$ denote the quadratic twist of $E$ over $F=\QQ(\sqrt{(-1)^{{(p-1)/2}}p})$, with $p$ odd. Assume that $E$ has additive reduction at $p$ but $E^F$ has semistable reduction at $p$. Then the periods of $E$ and $E^F$ are related as follows:
    If $p\equiv 1 \pmod{4}$, then
    \[\Omega^+_{E^F} = u_1\sqrt{p}\Omega^+_{E},\]
    and if $p\equiv 3 \pmod{4}$, then 
    \[\Omega^-_{E^F} = u_2 c_\infty(E^F)\sqrt{p}\Omega^+_{E},\]
    where $u_1,u_2$ are powers of $2$ and $c_\infty(E^F)$ is the number of connected components of $E^F(\mathbb{R})$.
\end{theorem}
\begin{proof}
The result \cite[Corollary 2.6]{pal} gives the relation for the potentially good case. For the potentially multiplicative case, see Prop. 2.4 of \textit{op. cit.} and consider the change in $p$-adic valuations of the invariants $\Delta_{E^F}$ and $c_4(E^F)$ upon twisting over $F$.
\end{proof}
In the forthcoming proofs, we relate the $\lambda(\theta_{n,i}(E))$ to  $\lambda(\theta_{n,i+(p-1)/2}(E^F))$ for even $i$. The latter is well-behaved for large $n$ since there exists a $p$-adic $L$-function for $E^F$. 
    \begin{theorem}\label{quad}
        Let $E/\QQ$ be an elliptic curve with additive reduction at an odd prime $p$. Let $i$ be an even integer between $0$ and $p-2$.
        Assume that 
        \begin{itemize}
        \item the quadratic twist $E^F$ has either good ordinary or multiplicative reduction at $p$ and 
        \item the $\mu$-invariant of $L_p(E^F,\omega^{(p-1)/2+i}, T)$ is zero {and the $\mu$-invariant of $\theta_{n,i}(E)$ is non-negative when $n$ is sufficiently large.}
        \end{itemize}
         Let $\lambda(E^F, \omega^{{(p-1)/2+i}})$ denote the $\lambda$-invariant of $L_p(E^F, \omega^{{(p-1)/2+i}}, T)$. 
    Then, for $n$ sufficiently large, 
    \begin{align*}
    \mu(\theta_{n,i}(E)) &= 0, \\ 
    \lambda(\theta_{n,i}(E))&= \frac{(p-1)}{2}\cdot{p^{n-1}} 
 + \lambda(E^F, \omega^{{(p-1)/2+i}}).\end{align*}
    \end{theorem}
\begin{remark}
   Recall from the discussion in \S\ref{sec:potmult} that when $E$ has potentially multiplicative reduction, it necessarily achieves multiplicative reduction over a quadratic extension. Thus, Theorem~\ref{quad} gives us a formula for $\lambda(\theta_{n,i}(E))$ for all cases of potentially multiplicative reduction provided that the assumptions on the $\mu$-invariants hold.

    We also note that the integrality of the $p$-adic $L$-function attached to $E^F$ is not guaranteed \textit{a priori} since we normalise by the Néron periods, but our assumption on the $\mu$-invariant ensures we have an integral power series (otherwise we would have $\mu<0$). Similarly, the assumption on $\mu(\theta_{n,i}(E))$ is to ensure integrality. Alternatively, assuming $\mu(\theta_{n,i}(E))= \mu(L_p(E^F, \omega^{(p-1)/2+i}, T))$ for all large $n$ also gives us the same formula for the $\lambda$-invariant. 

\end{remark}
\begin{proof}
We give the proof when $i=0$ for notational convenience; the entire argument remains the same for a general even $i$. For a character $\chi$ on $G_n$, we have 
   \[L(E,\chi, 1) = L(E^F, \omega^{(p-1)/2}\chi, 1),\]
   where $\omega^{(p-1)/2}$ is the quadratic character corresponding to the quadratic extension $F/\QQ$. 
   By the interpolation property of Mazur--Tate elements, we have
   \begin{align*}
    \overline{\chi}(\theta_{n, 0}(E)) &= \tau(\overline{\chi})\frac{L(E, \chi, 1)}{\Omega_E^+},
\end{align*}
which can be rewritten as
\[\overline{\chi}(\theta_{n, 0}(E)) = {\frac{\tau(\overline{\chi})}{\tau(\omega^{(p-1)/2}\overline{\chi})}}\cdot {\frac{\Omega_{E^F}^{\epsilon'}}{\Omega_E^+}}\cdot\left(\tau(\omega^{(p-1)/2}\overline{\chi}) \frac{L(E^F,\omega^{(p-1)/2}{\chi}, 1)}{\Omega_{E^F}^{\epsilon'}}\right),\]
where $\epsilon'=(-1)^{(p-1)/2}$. 
{For a general even $i$, the same expression holds after replacing $\theta_{n, 0}$ by $ \theta_{n, i}$ and $\chi$ by $\chi\omega^i$. If $i$ is odd, then $\Omega_{E^F}^{-\epsilon'}$ would appear in place of $\Omega_{E^F}^{\epsilon'}$.}
The ratio of the two Gauss sums is a $p$-adic unit (since $\omega^{(p-1)/2}\overline{\chi}$ and $\overline{\chi}$ have the same conductor when $n$ is large enough), and the ratio of the periods, up to $p$-adic units, is $\sqrt{p}$ by Theorem \ref{thm: pal}. Taking valuations on both sides gives  
\[\ord_p(\overline{\chi}(\theta_{n, 0}(E))) = \frac{1}{2}+ \ord_p\left(\tau(\omega^{(p-1)/2}\overline{\chi}) \frac{L(E^F,\omega^{(p-1)/2}{\chi}, 1)}{\Omega_{E^F}^{\epsilon'}}\right).\]
We focus on computing the valuation on the right-hand side. Crucially, we can attach a $p$-adic $L$-function to $E^F$ having the following interpolation property: 
\[L_p(E^F,\omega^{(p-1)/2}, \zeta_{p^n}-1)= \frac{1}{\alpha_{E^F}^{n+1}}\left(\tau(\omega^{(p-1)/2}\overline{\chi}) \frac{L(E^F,\omega^{(p-1)/2}{\chi}, 1)}{\Omega_{E^F}^{\epsilon'}}\right),\]  
where $\zeta_{p^n}$ is the image of a topological generator of $\Gamma$ under $\overline{\chi}$, and $\alpha_{E^F}$ is the root of the polynomial $X^2+a_p(E^F)X+p$ with trivial $p$-adic valuation when $E^F$ is ordinary at $p$ and it is $\pm1$ when $E^F$ is multiplicative at $p$.
This gives a formula for the valuation of $\overline{\chi}(\theta_{n, 0}(E))$, via the $p$-adic Weierstrass preparation theorem, in terms of the Iwasawa invariants of $L_p(E^F,\omega^{(p-1)/2}, T)$ for $n$ large enough:
\begin{equation}\label{ord1}
    \ord_p(\overline{\chi}(\theta_{n, 0}(E)))= \frac{1}{2} + \frac{\lambda(E^F, \omega^{(p-1)/2})}{p^{n-1}(p-1)}
\end{equation}
as we have assumed the $\mu$-invariant vanishes for this $p$-adic $L$-function.  We now compute $\ord_p(\overline{\chi}(\theta_{n, 0}(E)))$ differently as follows. For each $n$, define $\mu_n\colonequals\mu(\theta_{n,0}(E))$ and $\lambda_n\colonequals\lambda(\theta_{n,0}(E))$.  We write
\begin{align*}
    \theta_{n, 0}(E)(T)&=p^{\mu_n}g_n(T) u_n(T),\end{align*}
where $u_n(T)\in \Zp[[T]]^\times$ with $g_n(T)=T^{\lambda_n}+ p\cdot h_n(T) \in \Zp[T]$ for some $h_n(T) \in \Zp[T]$.
Hence,
\begin{align*}
    \ord_p(\overline{\chi}(\theta_{n, 0}(E))) &= \mu_n+ \ord_p(g_n(\zeta_{p^n}-1)).\end{align*}
Since $g_n(T)\in\Zp[T]$ is a $p$-distinguished polynomial, we have $\ord_p(g_n(\zeta_{p^n}-1))\ge0$. Furthermore, \eqref{ord1} implies that $0<\ord_p(\overline{\chi}(\theta_{n, 0}(E)))<1$ for $n$ sufficiently large. Since $\mu_n \in \ZZ_{\ge0}$, in order for  $0<\mu_n+\ord_p(g_n(\zeta_{p^n}-1))<1$, we must have $\mu_n=0$.

We have 
\begin{align*}
    \ord_p(\overline{\chi}(\theta_{n, 0}(E))) &\geq  \text{min}\left\{\frac{\lambda_n}{p^{n-1}(p-1)}, 1+\ord_p({h_n}(\zeta_{p^n}-1))\right\}.\end{align*}
    Using Equation \eqref{ord1}, we obtain, for $n\gg0$, 
    \begin{equation}\label{compare}
        \frac{1}{2} + \frac{\lambda(E^F, \omega^{(p-1)/2})}{p^{n-1}(p-1)}\geq \text{min}\left\{\frac{\lambda_n}{p^{n-1}(p-1)}, 1+\ord_p({h_n}(\zeta_{p^n}-1))\right\}.
    \end{equation}
For $n$ large enough, the left-hand side becomes strictly less than $1$, so 
\[1 > \text{min}\left\{\frac{\lambda_n}{p^{n-1}(p-1)}, 1+\ord_p(h_n(\zeta_{p^n}-1))\right\}.\]
Since $\ord_p(h_n(\zeta_{p^n}-1))\geq 0$ (as $h_n(T) \in \Zp[T]$), we deduce that $\frac{\lambda_n}{p^{n-1}(p-1)}<1$. Consequently, \eqref{compare} becomes an equality and 
\begin{equation}
    \frac{\lambda_n}{p^{n-1}(p-1)} = \frac{1}{2} + \frac{\lambda(E^F, \omega^{(p-1)/2})}{p^{n-1}(p-1)},
\end{equation}
which results in the desired formula for $\lambda_n$.
\end{proof} 

We now investigate the potentially supersingular case.
Recall from the statement of Theorem~\ref{thm:PW-ss} that we define
\[ q_n=\begin{cases}
    p^{n-1}-p^{n-2}+\cdots+p-1 \space \text{ if $n$ even}\\
    p^{n-1}-p^{n-2}+\cdots+p^2-p \space \text{ if $n$ odd.}
\end{cases} \]
Using a similar argument and the plus and minus $p$-adic $L$-functions defined in \cite{pollack03}, we have:

\begin{theorem}\label{ssquad}
    Let $E/\QQ$ be an elliptic curve with additive reduction at an odd prime $p$. Let $i$ be an even integer between $0$ and $p-2$. Assume that
    \begin{itemize}
        \item the quadratic twist $E^F$ has supersingular reduction at $p$ with $a_p(E^F)=0$ and 
        \item the $\mu$-invariants of the $\omega^{(p-1)/2+i}$-isotypic component of the plus and minus $p$-adic $L$-functions are both 0, that is, $\mu(L^\pm_p(E^F, \omega^{(p-1)/2+i}, T)) = 0$ {and the $\mu$-invariant of $\theta_{n,i}(E)$ is non-negative when $n$ is sufficiently large.}
    \end{itemize} 
    Let $\lambda^\pm(E^F, \omega^{(p-1)/2+i})$ denote the $\lambda$-invariants of $L^\pm_p(E^F, \omega^{(p-1)/2+i}, T)$ respectively. Then we have, for all $n$ large enough, 
    \begin{align*}
      \mu(\theta_{n,i}(E)) &= 0, \\ 
   \lambda(\theta_{n,i}(E))&= \frac{(p-1)}{2}\cdot p^{n-1} 
 + q_n+ \begin{cases} \lambda^+(E^F, \omega^{(p-1)/2+i}) \text{  if $n$ even,}\\
  \lambda^-(E^F, \omega^{(p-1)/2+i})
  \text{  if $n$ odd}.\end{cases}
 \end{align*}
\end{theorem}
\begin{remark}
  A well-known conjecture of Greenberg (\cite[Conjecture~1.11]{greenberg}) asserts that for a good ordinary prime $p$, the $\mu$ invariant of the $p$-adic $L$-function is always zero when $E[p]$ is irreducible as a $\Gal(\overline{\QQ}/\QQ)$-representation.
   In the good supersingular case, it is conjectured (\cite[Conjecture~6.3]{pollack03}) that both $\mu(L^\pm_p(E^F, \omega^{i}, T))=0$ for all $0\leq i\leq p-2$. Both conjectures are supported by a large amount of numerical data. 
\end{remark}
\begin{proof}
    One proceeds as in the proof of Theorem \ref{quad}. The only difference is that we relate the Mazur--Tate elements of $E^F$ to the plus and minus $p$-adic $L$-functions via \cite[Proposition~6.18]{pollack03}. Indeed, we have
    \[\ord_p(\overline{\chi}(\theta_{n, 0}(E))) = \frac{1}{2}+\frac{ q_n + \lambda^\pm(E^F, \omega^{(p-1)/2+i})}{p^{n-1}(p-1)},\]
    where the sign is chosen according to the parity of $n$ (see Theorem \ref{thm:PW-ss}, \cite[Theorem 4.1]{PW}). 
    We write the analogue of equation \eqref{compare} and for large $n$, the inequality
    \[1 > \frac{1}{2}+ \frac{q_n + \lambda^\pm(E^F, \omega^{(p-1)/2+i})}{p^{n-1}(p-1)}{>0}\]
allows us to proceed as before to conclude the proof.  
\end{proof}
\subsection{Potential good ordinary reduction: The general case}
We give a more general result in the potentially good ordinary case, assuming {Conjecture \ref{conj:pBSD}}. This is Theorem \ref{thmB} in the introduction of this article.
\begin{theorem}\label{thm: bsd}
         Let $E/\QQ$ be an elliptic curve with additive, potentially good ordinary reduction at a prime $p\geq 5$ and minimal discriminant $\Delta_E$. 
         Assume that
        \begin{itemize}
            \item Conjecture~\ref{conj:pBSD} is true over $k_{n}$ for  $n \gg 0$,
            \item  $\lambda(\theta_{n,0}(E))<p^{n-1}(p-1)$ for $n\gg0$.
        \end{itemize}
Then, when $n$ is sufficiently large, we have 
        \begin{align*}
          \mu(\theta_{n,0}(E)) &= \mu(\cX(E/\QQ_\infty))\\
            \lambda(\theta_{n,0}(E)) &= \frac{(p-1)\cdot \ord_p(\Delta_E)}{12}\cdot p^{n-1}+{\lambda(\cX(E/\QQ_\infty))}.
        \end{align*}
\end{theorem}
\begin{proof}
We claim that we can choose $n$ large enough so that 
 \begin{enumerate}
     \item \label{Lval}$\ord_{s=1}L(E/k_{n+1}, s)=\ord_{s=1}L(E/k_{n}, s)$,
     \item \label{MW}$E(k_{n+1})=E(k_{n})$,
     \item \label{Tam}$\ord_p(\mathrm{Tam}(E/k_{n+1}))=\ord_p( \mathrm{Tam}(E/k_{n}))$.
 \end{enumerate} 
For (\ref{Lval}), from \cite{Rohrlich1984}, $L(E/\QQ, \chi, 1)=0$ for only finitely many Dirichlet characters $\chi$ of $p$-power order. We also have 
\[L(E/k_n,s)= \prod_{\chi}L(E/\QQ, \chi, s),\]
where the product is taken over all characters $\chi$ on $\Gal(k_n/\QQ)$. Thus, the order of vanishing of $L(E/k_n, s)$ at $s=1$ must stabilize for $n$ large. 

For (\ref{MW}), it follows from \cite[Theorem~14.4]{kato1} that $E(\QQ_\infty)$ is a finitely generated abelian group which makes this immediate. 

For (\ref{Tam}), recall from \cite[\S C, Corollary 15.2.1]{Si} that if $v$ is a prime of $k_n$, we have $C(E/k_{n,v})=-\ord_v(j(E))$ when $E$ has split multiplicative reduction at $v$; otherwise, $C(E/k_{n,v})\le 4$ and so $\ord_p(C(E/k_{n,v}))=0$. As $E$ is assumed to have potentially good ordinary reduction at $p$, if $v$ is the unique prime lying above $p$, then $\ord_p(C(E/k_{n,v}))=0$. For $v\nmid p$, if $\ell$ is the rational prime lying below $v$, then $\ell$ is unramified in $k_n$. Thus, $C(E/\QQ_\ell)=C(E/k_{n,v})$. Furthermore, the number of primes lying above $\ell$ is bounded independently of $n$. Therefore, $\ord_p(\mathrm{Tam}(E/k_n))$ stabilizes as $n$ increases.

For such $n$, {it follows from \eqref{bsdp} that}
\begin{equation}\label{pbsd}
    {\ord_p\left(\frac{|\Sha(E/k_{n+1})[p^\infty]|}{|\Sha(E/k_{n})[p^\infty]|}\right)= \ord_p\left(\frac{\Omega_{E/k_{n}}}{\Omega_{E/k_{n+1}}}{ \prod_{\chi}L(E/\QQ, \chi, 1)}\cdot\sqrt{\frac{|\Delta_{k_{n+1}}|}{|\Delta_{k_{n}}|}}\cdot \frac{\textup{Reg}(E/k_{n})}{\textup{Reg}(E/k_{n+1})}\right)},
\end{equation} 
where the product is taken over all characters $\chi$ on $G_{n+1}$ that do not factor through $G_n$. There are $p^{n}(p-1)$ such characters, each of conductor $p^{n+2}$. The conductor-discriminant formula tells us that
\[\ord_p\left(\sqrt{\frac{|\Delta_{k_{n+1}}|}{|\Delta_{k_{n}}|}}\right)= p^{n}(p-1)\cdot \frac{n+2}{2}.\]
For any finite extension of number fields $L_1/L_2$ such that $E(L_1)=E(L_2)$, we have 
\[\frac{\textup{Reg}(E/L_1)}{\textup{Reg}(E/L_2)}=[L_1:L_2]^{\text{rank}(E(L_2))}.\]
Thus, for $L_1/L_2=k_{n+1}/k_n$, the quotient of the regulators is equal to $p^{\text{rank}(E(k_{n+1}))}$. 
We deduce from \eqref{perratio} that
\begin{equation}
    \frac{\Omega_{E/k_{n+1}}}{\Omega_{E/k_{n}}}= \Omega_{E/\QQ}^{p^n(p-1)} \prod_{v, w\mid v} \left|\frac{\omega_v^{\min}}{\omega_w^{\min}}\right|_{w},
\end{equation}
where $v$ runs over {the places of $k_{n}$ and $w$ over the places of $k_{n+1}$ lying} above $v$. Putting these together, the $p$-adic valuation of the right-hand side of \eqref{pbsd} can be expressed as
\begin{equation*}
     \ord_p\left(\prod_\chi\frac{L(E/\QQ, \chi, 1)}{\Omega_{E/\QQ}} \right)- \ord_p\left(\prod_{v, w\mid v} \left|\frac{\omega_v^{\min}}{\omega_w^{\min}}\right|_{w}\right) + p^{n}(p-1)\cdot \frac{n+2}{2}- \text{rank}(E(k_{n+1})).
\end{equation*}

Applying Lemma \ref{fudge}(2) on the totally ramified extension $(k_{n+1})_p/(k_n)_p$ gives
\[\ord_p\left(\prod_{v, w\mid v} \left|\frac{\omega_v^{\min}}{\omega_w^{\min}}\right|_{w}\right)=\left\lfloor \frac{p^n(p-1)\ord_p(\Delta_{E})}{12}\right\rfloor.\]
The floor function can be omitted since the quantity $(p^n(p-1)\ord_p(\Delta_{E}))/{12}$ is an integer in the potentially good ordinary case (this follows from the definition of the semi-stability defect (Equation \eqref{eq: semistabilitydef}) and Lemma \ref{ediv}). From the interpolation property of Mazur--Tate elements, we have
\[{\ord_p\left(\tau(\overline{\chi})\cdot\frac{L(E/\QQ, \chi, 1)}{\Omega_{E/\QQ}} \right)}= \mu(\theta_{n+1,0}(E))+ \frac{\lambda(\theta_{n+1,0}(E))}{p^{n}(p-1)}.\]
(This is where we use the assumed upper bound on the $\lambda$-invariants to compute the valuation.) Since the Gauss sums satisfy $\tau(\chi)\tau(\overline{\chi})=\pm p^{n+2},$ we have
\[\ord_p\left(\prod_\chi \tau(\chi)\right)=p^n(p-1)\cdot \frac{n+2}{2}.\]
On the left hand side of \eqref{pbsd}, we know from Theorem \ref{sha} that for sufficiently large $n$, we have
\[{\ord_p\left(|\Sha(E/k_{n+1})[p^\infty]|\right)-\ord_p\left(|\Sha(E/k_{n})[p^\infty]|\right)}= \mu_E\cdot p^n(p-1) + \lambda_E - \text{rank}(E(\QQ_\infty)),\]
where $\mu_E,\lambda_E$ denote $\mu(\cX(E/\QQ_\infty))$ and $\lambda(\cX(E/\QQ_\infty))$ respectively.
Hence, we deduce that
\[\lambda(\theta_{n+1,0}(E))= (\mu_E-\mu(\theta_{n+1,0}(E)))\cdot p^n(p-1) + \left( \frac{p^n(p-1)\cdot\ord_p(\Delta_{E})}{12}\right)+\lambda_E.\]
Now, if $\mu_E> \mu(\theta_{n+1,0}(E))$, then $\lambda(\theta_{n+1,0}(E))>p^n(p-1)$ which is false under our assumption. If for some $n$ sufficiently large, $\mu_E<\mu(\theta_{n+1,0}(E))$, then $\lambda((\theta_{n+1,0}(E))<0$, which is not possible. Therefore, we must have $\mu_E=\mu(\theta_{n+1,0}(E))$ for large $n$, which gives the formula for $\lambda(\theta_{n+1,0}(E))$, as desired. 
\end{proof}
\subsubsection{Relation to Delbourgo's $p$-adic $L$-functions}\label{sssec: Delbourgo's Lp}
We discuss Delbourgo's construction of $p$-adic $L$-functions in \cite{del-compositio} for additive primes briefly to highlight that our proofs of Theorems \ref{quad} and \ref{ssquad} are in fact closely related to the aforementioned work. Delbourgo's construction is based on considering the newform obtained by twisting $f_E$ by a power of $\omega$. This twist has a non-trivial Euler factor at $p$, for which the construction of a $p$-adic $L$-function is possible. The potentially good ordinary case satisfies Hypothesis (G) therein (see \cite[Section 1.5]{del-compositio}) since good reduction is achieved over an abelian extension of $\Qp$ with degree dividing $(p-1)$ (see Lemmas \ref{lem: Kgdeg} and \ref{ediv}). In the potentially multiplicative case, Delbourgo considered the quadratic twist as we have done above. See also \cite[Lemma 1]{Bayer1992}, which explains how a twist of $f_E$ can be used to obtain a $p$-ordinary newform with level $N$ such that $p\mid\mid N$ when $E$ has potentially good ordinary reduction at $p$.

Let $e$ be the semistability defect, and let $\tilde{f}$ be the newform $f_E \otimes \omega^{(p-1)/e}$, where $E$ is an elliptic curve over $\QQ$ with potentially good ordinary reduction. To ensure that one of the roots of the Hecke polynomial of $\tilde{f}$ is a $p$-adic unit, one may need to twist $f_E$ by $\omega^{-(p-1)/e}$ instead. In what follows, we assume that the twist is by $\omega^{(p-1)/e}$. Consider the $\omega^{-(p-1)/e}$-isotypic component of the $p$-adic $L$-function $L_p(\tilde{f},\omega^{-(p-1)/e},T)$ (the corresponding Mazur--Tate element will be $\theta_{n, -(p-1)/e}(\tilde{f})$). An argument similar to the proof of Theorem~\ref{quad}, assuming integrality and the vanishing of the $\mu$-invariant, would give formulae of the form
\begin{align}\label{eqn: delb}
    \lambda(\theta_{n,0}(E)) = \ord_p\left(\frac{\Omega_{\tilde{f}}}{\Omega_E}\right)\cdot p^{n-1}(p-1)+\lambda(L_p(\tilde{f},\omega^{-(p-1)/e},T)).
\end{align}
Note that the assumption on $\mu(L_p(\tilde{f},\omega^{-(p-1)/e},T))$ would pin down a period for $\tilde{f}$. We note that Theorem~\ref{thm: bsd} gives an expression for the $p$-adic valuations of the ratio of periods that occurs in \eqref{eqn: delb} in terms of $\ord_p(\Delta_E)$.

We also mention that when $E$ has additive potentially supersingular reduction with $e=2$, we can use the plus and minus $p$-adic $L$-functions of $E^F$ to construct the counterparts for $E$. Delbourgo commented that his construction would give unbounded $p$-adic $L$-functions. Indeed, if we combine Delbourgo's construction with the work of \cite{pollack03}, one may decompose these unbounded $p$-adic $L$-functions into bounded ones, resulting in the elements utilized in the proof of Theorem~\ref{ssquad}.
\subsection{Speculation for the potentially supersingular case}\label{ss: speculation for pot. ss.}
For elliptic curves with additive potential supersingular reduction at a prime $p\geq3$, the extension over which supersingular reduction is achieved is of degree $e$ dividing $p+1$ and is not Galois. Taking its Galois closure gives us a non-abelian extension of the form $\QQ(\zeta_m, \sqrt[e]{p})$ for $m=3$ or 4. Unlike the potentially good ordinary case, the growth in the $\lambda$-invariants cannot be attributed solely to the valuation of the ratio of periods, as will become clear in the following example. 

\subsubsection{An example}

\begin{example}\label{example}
    Consider the elliptic curve $E=$\href{https://www.lmfdb.org/EllipticCurve/Q/4232i1/}{4232i1} at $p=23$. This curve achieves good supersingular reduction at a prime above $p$ over an extension of degree $e=6$. The $\lambda$-invariants of the associated Mazur--Tate elements are seen to satisfy, for $n\leq 79$,
    \[\lambda(\theta_{n,0}(E))= 19\cdot 23^{n-1}+1.\]
    Assuming the validity of the Birch and Swinnerton-Dyer conjecture, one can subtract off the growth coming from the periods (as given by Proposition~\ref{fudge}(2)) to deduce 
    \[\ord_p(\Sha(E/k_{n})[p^\infty])-
    \ord_p(\Sha(E/k_{n-1})[p^\infty])= \begin{cases}
        16q_{n-1}+17 \text{ for $n$ even,} \\ 16q_{n-1}+1 \text{ for $n$ odd.}
    \end{cases}\]

    The method in the proof of Theorem \ref{ssquad} does not apply in this case. To the authors' knowledge, such growth for $\Sha$ cannot be explained by the currently known plus and minus Iwasawa theory, which is what we would need for a proof similar to that of Theorem \ref{thm: bsd}. 
\end{example}

In Example~\ref{example}, $\lambda(\theta_{n,0}(E))$ is predicted to be given by $19\cdot 23^{n-1}+1$ according to our numerical data.  In particular, it does not exhibit the alternating plus/minus behaviour as in the good supersingular case; the valuation coming from the period ratios (applying Equation \eqref{perratio} to $(k_{n})_p/(k_{n-1})_p$ like we did in the proof of Theorem \ref{thm: bsd}) is given by 
\begin{align*}
    f_n\colonequals\left\lfloor \frac{p^{n-1}(p-1)\ord_p(\Delta_{E})}{12}\right\rfloor&=\left\lfloor \frac{(p+1)(p-1)\ord_p(\Delta_{E})}{e\cdot\gcd(12,\ord_p(\Delta_E))}\cdot\frac{p^{n-1}}{p+1}\right\rfloor\\
&=\left(\frac{(p^2-1)\cdot\ord_p(\Delta_{E})}{e\cdot\gcd(12,\ord_p(\Delta_E))}\right)\cdot q_{n-1}+O(1),
\end{align*}
where we use the formula for the semi-stability defect $e=12/\gcd(12,\ord_p(\Delta_E))$ (Equation \eqref{eq: semistabilitydef}). Observe that $\frac{p+1}{e}$ is an integer by Lemma \ref{ediv}. Then, the term $p^{n-1}/(p+1)$ matches $q_{n-1}$ upto a difference bounded with $n$, see for example \cite[Lemme 5.4]{PR}.

When $f_n$ is added to the valuation of the size of $\Sha$, the contribution of $q_{n-1}$ seems to be eliminated. We do not have a theoretical explanation for this phenomenon at present. It is interesting to note that in the potentially good ordinary case, the $f_n$ term has no such alternating behaviour since $e\mid p-1$.

We direct the reader to Table \ref{tab:pot.ss} given at the end of the article for further examples. 

\subsubsection{A conjecture of Lario}
In \cite{lario1992serre}, a conjecture involving mod $p$ Galois representations of elliptic curves that have additive potentially supersingular reduction at $p$ is formulated. In particular, it has been shown that this conjecture implies that Serre's modularity conjecture holds for the mod $p$ representation $\overline{\rho}_E$ (where $p$ is a potentially supersingular prime for $E$). Here, we show that Serre's conjecture (which is now a theorem by works \cite{khare2009serre} of Khare and Wintenberger for such $E$) implies the conjecture of Lario. {We use $\omega$ to denote the mod $p$ cyclotomic character.}
\begin{theorem}\label{conj: lario}
    Let $E$ be an elliptic curve with potentially supersingular reduction at an odd prime $p$ {such that the $G_\QQ$-representation $E[p]$ is irreducible}. Let $N_E=Np^2$ be the conductor of $E$. Let $c_4, c_6$ denote the usual invariants from a minimal Weierstrass equation for $E$. Let $e$ be the semistability defect of $E$ at $p$ and assume $e>2$.
         \begin{itemize}
  \item[(a)] If the Kodaira--N\'eron type is II, III, IV, II$^\star$ with $\ord_p(c_4)=4$,  III$^\star$ with $\ord_p(c_6)=5$, or IV$^\star$ with $\ord_p(c_4)=3$, then there exists a newform $G \in S_2(\Gamma_0(pN), \omega^{-2\left(\frac{p+1}{e}\right)})$ such that 
    \[a_\ell(G) \equiv a_\ell(E) \cdot \ell^{-\frac{p+1}{e}} \pmod{p}\]
    for all primes $\ell \nmid Np$.
\item[(b)] If the Kodaira--N\'eron type is II$^\star$ with $\ord_p(c_4)>4$, III$^\star$ with $\ord_p(c_6)>5$, or IV$^\star$ with $\ord_p(c_4)>3$, then there exists a newform $G \in S_2(\Gamma_0(pN), \omega^{2-2\left(\frac{p+1}{e}\right)})$ such that 
    \[a_\ell(G) \equiv a_\ell(E) \cdot \ell^{1-\frac{p+1}{e}} \pmod{p}\]
    for all primes $\ell \nmid Np$.
        \end{itemize}
\end{theorem}
\begin{proof}
  The argument relies on \cite[Propositions 1.2.2 and 1.2.3]{AScoh1986} on the existence of lifts of systems of mod $p$ eigenvalues of modular forms to characteristic zero. Let $\overline\rho_E: G_\QQ \to \text{Aut}(E[p])$ denote the mod $p$ representation attached to $E$. 
  
  First, assume the Kodaira-Neron type to be one of the types listed in (a). From \cite[\S~1, Page 74]{Bayer1992}, the Serre weight of $\overline\rho_E$ is given by 
    \[k_{\overline\rho_E} = 2+ \left(1+\frac{p+1}{e}\right)(p-1).\] 
        Consider the mod $p$ representation $\rho'\colonequals {\overline\rho_E} \otimes \omega^{-\frac{p+1}{e}}$. Twisting by $\omega^{-1}$ reduces the Serre weight by $p+1$. So, the Serre weight of $\rho'$ is given by 
    \[k_{\rho'}= p+1 - \frac{2(p+1)}{e}.\]
    Thus, by Serre's conjecture, there exists a cusp form $g \in S_{k_{\rho'}}(\Gamma_0(N))$ lifting the residual representation $\rho'$. In particular, this system of eigenvalues (mod $p$) occurs in $H^1(\Gamma_0(N), V_{k_{\rho'}-2}(\Fp))$ and, by \cite[Theorem 3.4]{ash-ste}, this system of eigenvalues occurs in $H^1(\Gamma_0(pN),\Fp)^{\omega^{-2\left(\frac{p+1}{e}\right)}}$. From \cite[Prop. 1.2.3]{AScoh1986} and \cite[Prop. 2.5, Lemma 2.6]{ash-ste}, there exists $G \in S_2(\Gamma_0(pN), \omega^{-2\left(\frac{p+1}{e}\right)})$ lifting the system of eigenvalues corresponding to $\rho'$.

For Kodaira-Neron types listed in (b), the Serre weight is given by 
\[k_{\overline\rho_E} = 2+ \left(\frac{p+1}{e}\right)(p-1).\] 
Consider $\rho'\colonequals\overline\rho_E\otimes \omega^{1- ((p+1)/e)}$. Then 
\[k_{\rho'}=2+ \frac{p+1}{e}(p-1) - \frac{p+1}{e}(p+1)+{p+1}= p+3- 2\left(\frac{p+1}{e}\right)\leq p+1.\]
Thus, there exists a cusp form $g \in S_{k_{\rho'}}(\Gamma_0(N))$ with residual representation $\rho'$. The corresponding system of mod $p$ Hecke eigenvalues occurs in $H^1(\Gamma_0(N), V_{k_{\rho'}-2}(\Fp))$ and, by \cite[Theorem 3.4]{ash-ste} (as ${k_{\rho'}-2}<p$), occurs in $H^1(\Gamma_0(pN),\Fp)^{\omega^{2-2\left(\frac{p+1}{e}\right)}}$. As before, using \cite[Prop. 1.2.3]{AScoh1986} and \cite[Prop. 2.5, Lemma 2.6]{ash-ste}, we deduce that there exists $G \in S_2(\Gamma_0(pN), \omega^{2-2\left(\frac{p+1}{e}\right)})$ lifting the system of eigenvalues corresponding to $\rho'$.
\end{proof}
\begin{remark}
\begin{itemize}
    \item When $p$ is a prime of potentially good ordinary reduction, there exists a newform in $S_2(\Gamma_0(pN),\omega^{-\frac{2(p-1)}{e}})$ having the same system of eigenvalues as the twist $f_E\otimes \omega^{-\frac{p-1}{e}}$ in characteristic zero (see \cite{Bayer1992}[Pg. 77, Lemma 1]). This is utilised by Delbourgo in \cite{del-compositio} to construct a $p$-adic $L$-function for $E$. In comparison, when $p$ is a potentially supersingular prime, Theorem~\ref{conj: lario} gives an analogous statement for $f_E$ `mod $p$'. This suggests that some form of mod $p$ multiplicity one result would be required to translate the congruence of residual representations into a congruence involving modular symbols of $E$. 
\end{itemize} 
\end{remark}

\subsubsection{Some CM curves}
For {CM} elliptic curves with additive potentially supersingular reduction {at primes that are inert in the CM field}, computations on SageMath suggest the following conjecture (see the last 4 entries of Table \ref{tab:pot.ss}).
\begin{conjecture}\label{cm}
  Let $E/\QQ$ be a CM elliptic curve with additive, potentially supersingular reduction at $p \geq 5$. Assume that at least one of the following hypotheses hold:
    \begin{itemize}
        \item {the prime $p$ is inert in the CM field, or}
        \item {the $j$-invariant is either $0$ or $1728$.}
    \end{itemize}
    Then there exist integers
    $\mathcal{M},\lambda^+, \lambda^-, \nu$ depending only on $E$ and $p$, such that if $|\Sha(E/k_{n})[p^\infty]|= p^{e_n}$ then
    \[{e_n-e_{n-1}=\mathcal{M}(p^n-p^{n-1})+q_n+\lambda^\pm}\]
    for all $n$ large enough.
\end{conjecture}
\begin{remark}
    We computed the $\lambda$-invariants using SageMath and combined with the Birch and Swinnerton-Dyer conjecture to formulate Conjecture~\ref{cm}. This conjecture states that $\Sha(E/k_{n})[p^\infty]$ shows growth similar to what is known for elliptic curves which have good supersingular reduction at $p$ (see, for example, \cite[Section~6.4]{pollack03}). We can equivalently state this in terms of the $\lambda$-invariants of the Mazur--Tate elements after adding the growth from the periods (\eqref{perratio} and Proposition \ref{fudge}), but the above formulation makes the underlying `plus and minus' behaviour more apparent. However, the data suggest that curves that do not satisfy these hypotheses have remarkably different growth of $\Sha(E/k_n)[p^\infty]$ as compared to what is known in the literature, as seen in Example \ref{example}. 
\end{remark}
Let $E$ be as in Conjecture~\ref{cm}. Suppose $E$ has CM by an imaginary quadratic field $K$ and $p$ is inert in $K$. Let $O_K$ denote the ring of integers of $K$ and let $O_p$ denote the completion of $O_K$ at $p$. The $G_K$-representation attached to $E$ is given by the character
\[\rho_E: G_K \to \mathrm{Aut}_{O_p}(E[p^\infty])\cong O_p^\times.\]
Furthermore, since $E$ has CM, the quadratic twist of $E$ by $K$ (denoted by $E^K$) is isogenous to $E$ over $\QQ$. Thus,
\[
    L(E/K,\chi, s)=L(E/\QQ,\chi, s)^2
\]
for a character $\chi$ on $\Gal(K\QQ_\infty/K)\cong\Gal(\QQ_\infty/\QQ)$. 
Thus, assuming the Birch and Swinnerton-Dyer conjecture holds over $k_n$, Conjecture \ref{cm} is equivalent to showing that 
\[\ord_p\left({\tau(\overline{\chi})^2}\frac{L(E/K,\chi,1)}{\Omega_{E/\QQ}^2}\right) = \mathcal{M}'(p^n-p^{n-1})+2q_n+\lambda^{'\pm}\]
for constants $\mathcal{M}', \lambda'^\pm$ independent of sufficiently large $n$.

In a forthcoming joint work of Ben Forrás and the first-named author, we will study the plus and minus Selmer groups via generalized Honda theory for formal groups, applied to a particular family of CM elliptic curves. This will yield a resolution of Conjecture~\ref{cm} when $p$ is inert in the CM field under suitable hypotheses.

There are examples of CM elliptic curves with potentially supersingular reduction at $p$ with $p$ ramified in the associated CM field and $j \neq 0$ or $1728$ that do not satisfy the growth predicted by Conjecture \ref{cm}. For instance, consider 
$(E,p)= (\href{https://www.lmfdb.org/EllipticCurve/Q/441d1/}{441d1}, 7)$. In this case, the CM discriminant is $-7$, the $j$-invariant $-3375$, $\ord_p(\Delta_E)=3$ and the semi-stability defect $e=4$. The $\lambda$-invariants of the associated Mazur--Tate elements are seen to satisfy, for all $n\leq 7$,  
\[\lambda(\theta_{n,0}(E))=2\cdot7^{n-1}.\]
Assuming the Birch and Swinnerton-Dyer conjecture, one can deduce (as in Example~\ref{example}, see also \S~\ref{pss data}) that 
\[e_n-e_{n-1}= \begin{cases}
    4q_{n-1}+  3\text{ for $n$ even,}\\
    4q_{n-1}+ 1\text{ for $n$ odd,}\\
\end{cases}\]
where $p^{e_n} =|\Sha(E/k_n)[p^\infty]|$.

\section{Generalisations to potentially ordinary modular forms}\label{sec: results on modularforms}
In this section, we prove Theorem \ref{thmC} from the Introduction. We also discuss a generalisation of Theorem \ref{thm: bsd} to modular forms under appropriate hypotheses. We begin by recalling a natural generalisation of the fact that $\lambda(\theta_n(E/\QQ)) \geq p^{n-1}$ at additive primes as discussed in section \ref{sec:known}.
Throughout this section, we will consider a Hecke eigenform $f \in S_k(\Gamma_1(N))$ (i.e., $f$ is an eigenvector for all the Hecke operators $T_l$ for $l\nmid N$ and $U_q$ for $q \mid N$) and assume that $f$ is new at $p$.
\begin{lemma} \label{lem:lambdabound}
    Let $f$ be a Hecke eigenform in $S_k(\Gamma_1(N))$ with $p^2\mid N$ and $a_p(f)=0$. Then 
    \[\lambda(\theta_{n,i}(f))\geq p^{n-1}\]
    for all $0\leq i\leq p-2$.
\end{lemma}

\begin{proof}
    The argument of \cite[Corollary~5.3]{doyon-lei2} goes through as the local Euler factor of $f$ at $p$ is trivial. 
\end{proof}

Let $K_f$ denote the Hecke field of $f$. Fix an odd prime $p$. Fix a place $v$ above $p$ in $K_f$. Let $V_f$ denote the two-dimensional $p$-adic representation of $G_{\QQ}$ associated to $f$ over the completion $(K_f)_v$. 
We make the following assumption on $V_f$:
\vspace{-5pt}
\begin{itemize}
    \item[(\mylabel{item_Ord}{\textbf{PotOrd}})] $V_f$ admits a one-dimensional unramified $G_{\Qp(\mu_p)}$-subrepresentation $V_f'\subset V_f$.
\end{itemize}
We will assume $V_f$ satisfies (\textbf{PotOrd}) throughout this section.
\begin{remark}
   When $f$ is associated with an elliptic curve with additive reduction at $p$, this corresponds to the notion of potentially ordinary reduction, where the ordinary reduction is achieved over the abelian extension $\Qp(\mu_p)$. Recall that in the case of elliptic curves, \textit{all} elliptic curves having potentially ordinary reduction at a prime $p>3$ achieve ordinary reduction over $\Qp(\mu_p)$ (here ordinary refers to good ordinary or multiplicative reduction). CM modular forms with level divisible by $p$ also provide a large class of examples satisfying (\textbf{PotOrd}). This assumption is motivated by \cite[Remark 12.7]{kato1} and \cite[\S~3]{muller24}.
 \end{remark}
The next two lemmas are adapted from \cite[Lemmas~5.3 and 5.6]{muller24}. While the cited article focuses on CM forms, the results we discuss here apply to all $f$ satisfying (\textbf{PotOrd}).
\begin{lemma}\label{lem: V_f alpha lemma}
    The one-dimensional subspace $V'_f$ is $G_{\Qp}$-stable. Moreover, there exists a character $\psi$ of $\Gal(\Qp(\mu_p)/\Qp)$ such that $V'_f\otimes\psi^{-1}$ is an unramified $G_{\Qp}$-representation. 
\end{lemma}
\begin{proof}
The quotient $V''_f\colonequals V_f/V'_f$ is ramified as a $G_{\Qp(\mu_p)}$-representation since $V_f'$ is unramified and the determinant of $V_f$ is ramified. Let $I \subset G_{\Qp(\mu_p)}$ be the inertia subgroup of $G_{\Qp(\mu_p)}$. Then, $V'_f$ is the unique non-trivial subspace of $V_f$ that is fixed by $I$. Since $I$ is a normal subgroup of $G_{\Qp}$, we see that $V'$ is stable under $G_{\Qp}$.  

Under (\textbf{PotOrd}), the action of $G_{\Qp}$ on $V'_f$ factors through $\Gal(\Qp^{\mathrm{ur}}(\mu_p)/\Qp)$ since it is unramified as a $G_{\Qp(\mu_p)}$-representation. Consider the decomposition 
\[\Gal(\Qp^{\mathrm{ur}}(\mu_p)/\Qp)\cong \Gal(\Qp^{\mathrm{ur}}/\Qp)\times \Gal(\Qp(\mu_p)/\Qp),\]
and note that $\Gal(\Qp(\mu_p)/\Qp)$ acts on $V'_f$ through a character $\psi$ of $\Gal(\Qp(\mu_p)/\Qp)$. Consequently, the $G_{\Qp}$-action on $V'_f\otimes \psi^{-1}$ factors through $\Gal(\Qp^{\mathrm{ur}}/\Qp)$, so $V'_f\otimes \psi^{-1}$ is an unramified representation of $G_{\Qp}$, as desired.
\end{proof}

Let $f=\sum_{n}{a}_n(f)q^n$ and let $\tilde{f}=\sum_{n}{a}_l(\tilde{f})q^n$ be the newform corresponding to the twist of $f$ by 
$\psi^{-1}$, i.e., $a_l(\tilde{f})=a_l(f)\psi(l)^{-1}$ for all primes $l$ such that $(l,p)=1$. Such an $\tilde{f}$ is unique since a newform is uniquely determined by all but finitely many of its Hecke eigenvalues (strong multiplicity one).
Let ${V}_{\tilde{f}}$ denote the $G_\QQ$-representation associated with $\tilde{f}$. 
\begin{lemma}\label{lem: ap is a unit}
    The Fourier coefficient $a_p(\tilde{f})$ is a $p$-adic unit, and thus the Euler factor of $L(\tilde{f},s)$ at $p$ is non-trivial.
\end{lemma}
\begin{proof}
First, note that ${V}'_{f} \otimes \psi^{-1}$ is a crystalline $G_{\Qp}$-representation as it is unramified.  Let $\alpha$ denote the eigenvalue of the action of $\varphi$ on $V'_f$, which has $p$-adic valuation 0.
Further, we have 
$$
P(x):=\det\left(1-\varphi x | \mathbb{D}_\mathrm{cris}(V_{\tilde{f}})\right)= \text{local~Euler~factor~of~}\tilde f\text{~at~}p.
$$
Since ${V}'_{f} \otimes \psi^{-1} \subseteq V_{\tilde{f}}$, we see that $\mathbb{D}_\mathrm{cris}({V}_{\tilde{f}})$ is either 1- or 2-dimensional. In the 1-dimensional case, $a_p(\tilde{f}) = \alpha$, which is a $p$-adic unit.  In the 2-dimensional case, let $\beta$ denote the other root of $P(x)$.  We have $\alpha + \beta = a_p(\tilde{f})$ and $\alpha \beta = \varepsilon(p){p^{k-1}}$, where $\varepsilon$ is the nebentype of $\tilde{f}$.  Since $\alpha$ is a $p$-adic unit, $\ord_p(\beta)={k-1}$, which implies $\ord_p(a_p(\tilde{f}))=0$, as desired.
\end{proof}

Lemma \ref{lem: ap is a unit} tells us that any forms satisfying \textbf{(PotOrd)} admit a twist (by a character of conductor {1 or  $p$}) with {the power of $p$ in the level at most one}. 

\begin{remark}
    Clearly, one can construct examples of $f$ satisfying \textbf{(PotOrd)} by taking a cusp form that is good ordinary at $p$ and twisting it by a character of conductor $p$. Moreover, Lemma \ref{lem: ap is a unit} tells us that \emph{all} possible examples of $f$ satisfying \textbf{(PotOrd)} must arise this way. Alternatively, under the local Langlands correspondence for $\mathrm{GL}(2)$ (for $l=p$, using the local-global compatibility result of Saito \cite{Saito1997ModularFA}), it is clear that $V_f$ satisfying \textbf{(PotOrd)} corresponds to ramified principal series with $PS(\chi_1, \chi_2)$ with $\chi_1,\chi_2$ ramified (of conductor at most $p$) or a twist of a Steinberg representation. When $f$ corresponds to an elliptic curve over $\QQ$, this condition corresponds to $E$ having potentially good ordinary (ramified principal series) or potentially multiplicative (twist of Steinberg) reduction at $p$.
\end{remark}
Thanks to Lemma~\ref{lem: ap is a unit}, we can attach a $p$-adic $L$-function $L_p(\tilde{f})$ to {the ordinary $p$-stabilization of $\tilde{f}$ (as constructed in \cite{MTT}),} and use an appropriate twist of $L_p(\tilde{f})$ to give a $p$-adic $L$-function for $f$. This is the same as the method of Delbourgo \cite{del-compositio} for elliptic curves with additive reduction at $p$, which we discussed earlier in \S\ref{sssec: Delbourgo's Lp}; the work of Müller \cite{muller24} extends this idea further to higher-weight CM modular forms and in fact proves Kato's main conjecture for these forms under appropriate hypotheses including \textbf{(PotOrd)}. 
{Let $\mathcal{O}$ be the ring of integers of $(K_f)_v$. We use the cohomological periods in the definition of $L_p(\tilde{f})$ so that its coefficients are defined over $\mathcal{O}$. }

Using $\tilde f$, we prove the following reminiscent of Theorem~\ref{quad}. This is Theorem \ref{thmC} from the Introduction.
\begin{theorem}\label{thm: modforms twist}
    Let $p$ be an odd prime and let $f \in S_k(\Gamma_1(N))$ be a $p$-new Hecke eigenform with $k\geq 2$, $p^2\mid N$.  Assume that 
    \begin{itemize}
        \item the $p$-adic $G_\QQ$-representation $V_f$  associated with $f$ satisfies {(\textbf{PotOrd})} and 
        \item the Iwasawa invariants satisfy $\mu(\theta_{n,i}(f))=0$ and $\lambda(\theta_{n,i}(f)){<} p^{n-1}(p-1)$ for $n\gg 0$.
    \end{itemize}
     Then, there exist constants $m$ and $\lambda_0$ dependent only on $f$ and $i$, with {$\frac{1}{p-1}\le m\le \frac{p}{p-1}$ and $\lambda_0\ge0$}, such that for all $n\gg 0$, 
    \[\lambda(\theta_{n,i}(f))= m\cdot(p-1)\cdot p^{n-1} + \lambda_0.\]
\end{theorem}
\begin{proof}
We only discuss the case where $i=0$ for simplicity, but the same argument works for all the $\omega^i$-isotypic components $\theta_{n,i}(f)$.
    
As $\psi$ is a character on $\Gal(\Qp(\mu_p)/\Qp)$, we can write it as a power of the Teichm\"uller character $\omega$. Suppose $\psi=\omega^j$ for $0\leq j\leq p-1$. For a character $\chi$ on $\Gal(k_n/\QQ)$, we have 
    \[L(f, \chi,1)= L(\tilde{f}, \omega^{j}\chi,1)\]
from the definition of $\tilde{f}$. From the interpolation property of $\theta_{n,0}(f)$, we have
\[\overline{\chi}(\theta_{n, 0}(f)) = {\frac{\tau(\overline{\chi})}{\tau(\omega^{j}\overline{\chi})}}\cdot {\frac{\Omega_{\tilde{f}}^{\epsilon'}}{\Omega_f^+}}\cdot\left(\tau(\omega^{j}\overline{\chi}) \frac{L(\tilde{f}, \omega^{j}\chi,1)}{\Omega_{\tilde{f}}^{\epsilon'}}\right),\]
where $\epsilon'=(-1)^{j}$.  
From the $p$-adic $L$-function of $\tilde{f}$, we get  
    \[L_p(\tilde{f}, \omega^{j},\zeta_{p^n}-1) = \frac{1}{\alpha^{n+1}}\left(\tau(\omega^{j}\overline{\chi}) \frac{L(\tilde{f},\omega^{j}{\chi}, 1)}{\Omega_{\tilde{f}}^{\epsilon'}}\right),\]
    where $\zeta_{p^n}$ is the image of a topological generator of $\Gamma=\Gal(\QQ_\infty/\QQ)$ under $\overline{\chi}$ and $\alpha$ is the root of the Hecke polynomial of $\tilde{f}$ at $p$, which was shown to have trivial $p$-adic valuation in Lemma \ref{lem: ap is a unit}.
    Recall that $L_p(\tilde{f}, \omega^i, T)\in \mathcal{O}[[T]]$ for all $0\leq i\leq p-1$. Following the argument used for the proof of Theorem \ref{quad}, we have for $n\gg 0$, 
    \[\ord_p(\overline{\chi}(\theta_{n, 0}(f))) = \ord_p\left(\frac{\Omega_{\tilde{f}}^{\epsilon'}}{\Omega_f^+}\right)+ \mu(L_p(\tilde{f}, \omega^{j},T))+\frac{\lambda(L_p(\tilde{f}, \omega^{j},T))}{p^{n-1}(p-1)}\]
    since the ratio of Gauss sums has trivial $p$-adic valuation. 
    For each $n$, let $\lambda_n\colonequals\lambda(\theta_{n,0}(f))$. Under our assumption on $\lambda_n$, we can compute $\ord_p(\overline{\chi}(\theta_{n,0}(f))$ using the $p$-adic Weierstrass preparation theorem to get
    \[\ord_p(\overline{\chi}(\theta_{n,0}(f)))= \frac{\lambda_n}{p^{n-1}(p-1)}{<1},\]
    (thanks to the assumptions $\mu(\theta_{n,0}(f))=0$ {and $\lambda_n<p^{n-1}(p-1)$). Recall from Lemma~\ref{lem:lambdabound} that $\lambda_n\le p^n-1$ and that $\lambda_n\ge p^{n-1}$ by . Thus,
    \[
    \frac{1}{p-1}\le\ord_p\left(\frac{\Omega_{\tilde{f}}^{\epsilon'}}{\Omega_f^+}\right)+ \mu(L_p(\tilde{f}, \omega^{j},T))+\frac{\lambda(L_p(\tilde{f}, \omega^{j},T))}{p^{n-1}(p-1)}< 1
    \]
    As $\lambda(L_p(\tilde{f}, \omega^{j},T))\ge0$, we see that
    \[
    \frac{1}{p-1}\le \ord_p\left(\frac{\Omega_{\tilde{f}}^{\epsilon'}}{\Omega_f^+}\right)+ \mu(L_p(\tilde{f}, \omega^{j},T))<1. 
    \]
    Therefore, setting $m=\ord_p\left(\frac{\Omega_{\tilde{f}}^{\epsilon'}}{\Omega_f^+}\right)+ \mu(L_p(\tilde{f}, \omega^{j},T))$ and $\lambda_0=\lambda(L_p(\tilde f,\omega^j,T))$ gives the desired result.
    }\end{proof}
\begin{remark}
   The assumptions on $\mu(\theta_{n,i}(f))$ and $\lambda(\theta_{n,i}(f))$ can be replaced with the assumptions $\mu(L_p(\tilde{f},\omega^j,T))=0$ and $\ord_p(\Omega_{\tilde{f}}^{\epsilon'}/{\Omega_f^+})<1$ to get the same result using the bounding argument from Theorem~\ref{quad}. We omitted this from the statement since this may not be easy to verify directly for a given example. 
\end{remark}

\section{Numerical data for the potentially supersingular case}\label{pss data}
In the following table, for an elliptic curve $E$ defined over $\QQ$ with additive potentially supersingular reduction at a prime $p$, $e$ is the semistability defect (see Equation \eqref{eq: semistabilitydef}) and \[f_n \colonequals \left\lfloor(p^{n}\ord_p(\Delta_E))/12\right\rfloor- \left\lfloor(p^{n-1}\ord_p(\Delta_E))/12\right\rfloor\] denotes the valuation coming from the ratio of periods as in Equation \eqref{perratio} and Proposition~\ref{fudge}(2). The formulae for $\lambda(\theta_{n,0}(E))$ are seen to hold for all $n\leq 6$ in all the cases. The Sage code used to obtain this data is available at \href{https://github.com/namanpratap2001/Lambda-invariants-of-MT-elts/blob/main/lambdainvs%20with%20example}{https://github.com/namanpratap2001/Lambda-invariants-of-MT-elts/blob/main/lambdainvs}. Our code is essentially the same as the one used in \cite{doyon-lei} (available at \href{https://github.com/anthonydoyon/Ramanujan-s-tau-and-MT-elts/blob/master/cong.sagews}{https://github.com/anthonydoyon/Ramanujan-s-tau-and-MT-elts/blob/master/cong.sagews}). These are based on the second-named author's original code available at \href{https://github.com/rpollack9974/Iwasawa-invariants}{https://github.com/rpollack9974/Iwasawa-invariants}. The raw data is available at \href{https://github.com/namanpratap2001/Lambda-invariants-of-MT-elts/blob/main/rawdata}{https://github.com/namanpratap2001/Lambda-invariants-of-MT-elts/blob/main/rawdata}.
\begin{table}[h]
    \centering
    \begin{tabular}{|c|c|c|c|c|} \hline 
         $(E, p)$&  $\ord_p(\Delta)$&  $e$ &  $\lambda(\theta_{n,0}(E))$&$\lambda(\theta_{n,0}(E))-f_n$ \\ \hline 
 (121c1,11)& 4& 3& $7\cdot11^{n-1}$&$ 4q_n+\begin{cases}
     4 & \text{if $n$ is odd}\\ 0 & \text{if $n$ is even}  
 \end{cases}$ \\ \hline 
 (968d1,11)& 2& 6& $2\cdot 11^{n-1}$&$4 q_{n-1}+\begin{cases}
     1 &  \text{if $n$ is odd}\\ 3 &  \text{if $n$ is even}
 \end{cases}$ \\ \hline  
         (2890h1, 17)&  2&  6&  $3\cdot17^{n-1}+1$&$6q_{n-1}+
         \begin{cases}
             2 &  \text{if $n$ is odd}\\ 6 &  \text{if $n$ is even}
         \end{cases}$ \\ \hline 
 (2890e1, 17)& 8& 3& $11\cdot17^{n-1}$&$6 q_{n-1}+\begin{cases}
     0 &  \text{if $n$ is odd}\\ 6 &  \text{if $n$ is even}
 \end{cases}$ \\ \hline 
 (2116b1,23)& 2& 6& $4\cdot23^{n-1}$&$8 q_{n-1}+\begin{cases}
     1 &  \text{if $n$ is odd}\\ 7 &  \text{if $n$ is even}
 \end{cases}$ \\ \hline 
 (4232i1, 23)& 10& 6& $19\cdot23^{n-1}+1$&$16 q_{n-1}+\begin{cases}
     1 &  \text{if $n$ is odd}\\ 17 &  \text{if $n$ is even}
 \end{cases}$ \\ \hline 
 (3844d1, 31)& 3& 4& $8\cdot31^{n-1}$&$16 q_{n-1}+\begin{cases}
     1 &  \text{if $n$ is odd}\\ 15 &  \text{if $n$ is even}
 \end{cases}$ \\ \hline 
 (1089a1, 11)&  4&  3&  $4\cdot 11^{n-1}+ 3q_{n-1}+ \begin{cases}
     1 &  \text{if $n$ is odd}\\ 3 &  \text{if $n$ is even}
 \end{cases}$&$q_n+\begin{cases}
     2 &  \text{if $n$ is odd} \\ 0 &  \text{if $n$ is even}
 \end{cases}$ \\ \hline 
 (1089b1, 11)&  10&  6&  $9\cdot11^{n-1}+3 q_{n-1}+2$&$q_n+\begin{cases}
     2 &  \text{if $n$ is odd} \\ 0 &  \text{if $n$ is even}
 \end{cases}$ \\ \hline 
         (3872h2,11)&  3&  4&  $3\cdot11^{n-1}+5q_{n-1}+ \begin{cases}
             1 &  \text{if $n$ is odd} \\ 5 &  \text{if $n$ is even}
         \end{cases}$&$q_n+\begin{cases}
     2 &  \text{if $n$ is odd} \\ 0 &  \text{if $n$ is even}
 \end{cases}$ \\ \hline
 (4761b1,23)& 2& 6& $4\cdot23^{n-1}+15 q_{n-1}+\begin{cases}1& \text{if $n$ is odd}\\17 & \text{if $n$ is even} \end{cases}$&$q_n+2$ \\ \hline
    \end{tabular}
    \vspace{0.5cm}
    
    \caption{$\lambda$ invariants at potentially supersingular primes}
    \label{tab:pot.ss}
\end{table}

\newpage

\bibliographystyle{amsalpha}
\bibliography{references}
\end{document}